\documentclass[envcountsect,envcountsame]{amsart}

\usepackage{graphicx}
\usepackage{amssymb}
\usepackage{amsmath}%
\usepackage{amsfonts}%

\newtheorem{theoremintro}{Theorem}
\renewcommand\thetheoremintro{\Roman{theoremintro}}
\newtheorem{theorem}{Theorem}[section]
\newtheorem{lemma}[theorem]{Lemma}
\newtheorem{corollary}[theorem]{Corollary}
\newtheorem{proposition}[theorem]{Proposition}

\theoremstyle{definition}
\newtheorem{definition}[theorem]{Definition}
\newtheorem{example}[theorem]{Example}

\theoremstyle{remark}

\begin{document}

\title{Discrepancy and numerical integration on metric measure
spaces }


 \author[L. Brandolini]{Luca Brandolini}
\address{Dipartimento di Ingegneria Gestionale, dell'Informazione e della Produzione,
  Universit\`a degli Studi di Bergamo,
  Viale Marconi 5, 24044 Dalmine (BG),
  Italy}
\curraddr{}
\email{luca.brandolinii@unibg.it}
\thanks{}

\author[W.W.L. Chen]{William W. L. Chen}
\address{Department of Mathematics,
Macquarie University,
Sydney, NSW 2109, Australia.}
\email{william.chen@mq.edu.au}
\thanks{}

\author[L. Colzani]{Leonardo Colzani}
\address{Dipartimento di Matematica e Applicazioni,
   Universit\`a degli Studi di Milano-Bicocca,
   Via R. Cozzi 55, 20125 Milano,
   Italy}
\curraddr{}
\email{leonardo.colzani@unimib.it}
\thanks{}

\author[G. Gigante]{Giacomo Gigante}
\address{Dipartimento di Ingegneria Gestionale, dell'Informazione e della Produzione,
  Universit\`a degli Studi di Bergamo,
  Viale Marconi 5, 24044 Dalmine (BG),
  Italy}
\curraddr{}
\email{giacomo.gigante@unibg.it}
\thanks{}

\author[G. Travaglini]{Giancarlo Travaglini}
\address{Dipartimento di Matematica e Applicazioni,
   Universit\`a degli Studi di Milano-Bicocca,
   Via R. Cozzi 55, 20125 Milano,
   Italy}
\curraddr{}
\email{giancarlo.travaglini@unimib.it}
\thanks{}



\begin{abstract}
We study here the error of numerical integration on metric measure spaces
adapted to a decomposition of the space into disjoint subsets. We consider
both the error for a single given function, and the worst case error for all
functions in a given class of potentials. The main tools are the classical
Marcinkiewicz--Zygmund inequality and \textit{ad hoc} definitions of function spaces on metric measure spaces. 
The same techniques are used to prove the existence
of point distributions in metric measure spaces with small $L^p$ discrepancy with respect to certain classes 
of subsets, for example metric balls.
\end{abstract}

\date{}

\maketitle

\section{Introduction}
The starting point of this research is a simplified version of Lemma 5 in \cite{C} which gives an upper bound
of higher norms of the discrepancy of a random set of points in the unit square $[0,1]^2$, treated as a torus. 
Let $N=M^2$ and consider a random set of $N$ points $\mathcal P$ as follows: 
Split the unit square into $N$ small squares $\{S_j\}_{j=1}^N$ of area $N^{-1}$ in the usual way. In each small square $S_j$
there is a random point $x_j$, uniformly distributed in the small square, independently of the distribution
of all the other random points in the other small squares.

Suppose that $\mathcal B$ is a convex set in $[0,1]^2$. Let $\mathcal J$ denote the set af all values of $j$ for which 
the small squares $S_j$ intersect the boundary  $\partial \mathcal B$ of $\mathcal B$. Then it is easy to see
that the cardinality of $\mathcal J$, $|\mathcal J|$, is $O(M)$. For each $j\in\mathcal J$, write 
\[
\xi_j=\chi_{\mathcal B}(x_j)=\left\{\begin{array}{ll}
1 & \text{if }x_j\in\mathcal B,\\
0 & \text{otherwise},
\end{array}
\right.
\]
and let $\eta_j=\xi_j-\mathbb E\xi_j$. Then, $|\eta_j|\leqslant 1$ and $\mathbb E\eta_j=0$.
Furthermore if we define the discrepancy as
\begin{equation}
\label{discrepancy}
D[\mathcal P, \mathcal B]:=\frac 1 N\sum_{j=1}^N\chi_\mathcal B(x_j)-|B|,
\end{equation}
then
\[
D[\mathcal P, \mathcal B]=\frac 1N\sum_{j\in\mathcal J}\eta_j.
\]
We now want to estimate $\mathbb E(|D[\mathcal P,\mathcal B]|^p)$ from above,
where $p$ is an even positive integer. Note first that
\[
|D[\mathcal P,\mathcal B]|^p=\frac1{N^p}\sum_{j_1\in\mathcal J}\ldots\sum_{j_p\in\mathcal J}\eta_{j_1}\ldots\eta_{j_p}
\]
and so
\begin{equation}
\label{combinat}
\mathbb E(|D[\mathcal P,\mathcal B]|^p)=\frac1{N^p}\sum_{j_1\in\mathcal J}\ldots\sum_{j_p\in\mathcal J}\mathbb E(\eta_{j_1}\ldots\eta_{j_p})
\end{equation}
The random variables $\eta_j$, where $j\in\mathcal J$, are independent because the distribution of the random points are 
independent of each other. If one of $j_1,\ldots, j_p$, say $j_i$, is different from all the others, then 
\[
\mathbb E(\eta_{j_1}\ldots\eta_{j_p})=\mathbb E(\eta_{j_i})\mathbb E(\eta_{j_1}\ldots\eta_{j_{i-1}}\eta_{j_{i+1}}\ldots\eta_{j_p})=0.
\]
It follows that the only non-zero contribution to the sum \eqref{combinat} comes from those terms where each of $j_1,\ldots,j_p$ appears
more than once. It can be shown that the major contribution comes when they appear in pairs, and there are
\[
O_p\left({{|\mathcal J|}\choose{p/2}}\right)=O_p({|\mathcal J|}^{p/2})=O_p({M^{p/2}})=O_p({N^{p/4}})
\]
such pairs. Bounding each of such terms $\mathbb E(\eta_{j_1}\ldots\eta_{j_p})$ trivially by $1$, we obtain the estimate
\begin{equation}
\label{basic}
\mathbb E(|D[\mathcal P,\mathcal B]|^p)=O_p(N^{-3p/4}).
\end{equation}

This result is the first step towards the proof of the existence of point sets with small $L^p$ discrepancy with respect to,
say, all discs in the square. Indeed, let $B(x,r)$ be the ball centered in the point $x$ and with radius $r$. Then, an application of the above 
estimate to the sets $B(x,r)$ gives
\begin{equation}
\int_0^{1/2}\int_{[0,1]^2}\mathbb E(|D[\mathcal P,B(x,r)|^p)dx\,dr=O_p(N^{-3p/4})
\end{equation}
and Fubini's theorem immediately implies 
the existence of an $N$-point set $\mathcal P$ such that
\begin{equation}
\label{discs}
\left(\int_0^{1/2}\int_{[0,1]^2}\left|D[\mathcal P,B(x,r)]\right|^p\,dx\,dr\right)^{1/p}=O_p(N^{-3/4}).
\end{equation}
By the monotonicity of the $L^p$ norms, one obtains these estimates for all $p<+\infty$.

This argument can be easily extended to a very general setting. In some sense, all that one needs for the argument to work is
\begin{enumerate}
\item a partition of the ambient space into $N$ parts with the same measure and similar diameter (the analog of the ``small squares'' in the previous argument);
\item a collection of sets with uniformly regular boundary, in such a way that the cardinality of the collection of indices 
$\mathcal J$ can be controlled uniformly in terms of the diameter of the ``small squares''. 
\end{enumerate}
We will therefore be able to replace the unit square with a Riemannian manifold, or more generally,  with metric measure spaces $\mathcal M$ having finite measure, 
as far as we assume that for any integer $N$  this space can be partitioned as required. By a recent result \cite{G-L}  this can be done under very general hypotheses.

In fact, we can replace the characteristic function of the set $\mathcal B$ with more general functions, so that our results are actually results on numerical integration. 
Consider the integral
\[
\int_{\mathcal{M}}f\left(  x\right)  dx
\]
of a function $f\left(  x\right)  $ over the metric measure space $\mathcal{M}$
with finite measure $dx$ and distance between two points $x$ and $y$ denoted with
$|x-y|$,  and the Riemann sums
\[
\sum_{j=1}^{N}\omega_{j}f\left(  x_{j}\right)  ,
\]
where $\left\{  x_{j}\right\}  _{j=1}^{N}$ are nodes in $\mathcal{M}$ and
$\left\{  \omega_{j}\right\}  _{j=1}^{N}$ are given weights. We are interested
in the rate of decay of the error
\[
{\sum_{j=1}^{N}}\omega_{j}f\left(  x_{j}\right)  -{\int_{\mathcal{M}}}f\left(
x\right)  dx
\]
as $N\rightarrow+\infty$. This decay depends on the smoothness of the function
$f\left(  x\right)  $, the weights $\left\{  \omega_{j}\right\}  _{j=1}^{N}$
and the distribution of the nodes $\left\{  x_{j}\right\}  _{j=1}^{N}$ in
$\mathcal{M}$. For references on this problem when the metric space is a
torus, a sphere, or more generally a compact Riemannian manifold, see, for
example, \cite{BCCGST}, \cite{BChGT}, \cite{BCGT}, \cite{BH}, \cite{BSSW},
\cite{He}, \cite{HMS}, \cite{HS1}, \cite{HS2}, \cite{HS3}, \cite{K}. For some
results related to metric measure spaces, see \cite{MM}, \cite{Str1}, \cite{Str2}.

Here we proceed as in the situation described before
for the study of the discrepancy, and partition $\mathcal{M}$ into $N$
disjoint measurable sets $\mathcal{X}_{1},\cdots,\mathcal{X}_{N}$ with
positive measure, set $\omega_{j}=\left\vert \mathcal{X}_{j}\right\vert $,
where $\left\vert \cdot\right\vert $ denotes the measure, and consider random
choices of points $x_{j}\in\mathcal{X}_{j}$.

To fix the notation we write $\boldsymbol{\omega}=\left(  \omega_{1}%
,\ldots,\omega_{N}\right)  $, $\mathbf{x}=\left(  x_{1},\ldots,x_{N}\right)
$, $\mathbf{X}=\mathcal{X}_{1}\times\cdots\times\mathcal{X}_{N}$,
$d\mathbf{x}=\frac{dx_{1}}{\omega_{1}}\times\cdots\times\frac{dx_{N}}%
{\omega_{N}}, $ and consider the probability space $\left(  \mathbf{X}%
,d\mathbf{x}\right)  $. We also write the error as
\begin{equation}
\label{err}
\mathcal{E}_{\mathbf{x},\mathbf{\omega}}\left(  f\right)  ={\sum_{j=1}^{N}%
}\omega_{j}f\left(  x_{j}\right)  -{\int_{\mathcal{M}}}f\left(  x\right)  dx.
\end{equation}
Notice the analogy with the discrepancy \eqref{discrepancy}.
There will be however an important difference:  
In order to measure the smoothness of our functions we will use suitable Besov spaces or potential spaces
adapted to this more general context  and
obtain estimates of $\mathcal{E}_{\mathbf{x},\mathbf{\omega}}\left(  f\right)
$ for functions in such spaces. 
The previous combinatorial argument, however, works only when the integrability exponent $p$
is an even integer, whereas in this case in order to obtain sharp results, we need estimates that work for 
generic values of $p$. 
The main idea is to replace such combinatorial argument with a generalization of the classical Khintchine inequality for sums 
of random variables, due to Marcinkiewicz and Zygmund \cite{MZ1,MZ2}. It says that for every $1\leqslant p <+\infty$ and for  
every sequence of independent random variables $\{f_j\}_{j=1}^N$
\[
 \mathbb{E}\left(  \left\vert {\displaystyle\sum_{j=1}^{N}}
\left(  f_{j}-\mathbb{E}\left(  f_{j}\right)  \right)  \right\vert
^{p}\right) \approx_p  \mathbb{E}\left(\left(
{\displaystyle\sum_{j=1}^{N}} \left\vert f_{j}-\mathbb{E}\left(  f_{j}\right)
\right\vert ^{2}\right)  ^{p/2}\right).
\]
For the case of discrepancy, this immediately gives 
\begin{equation*}
\mathbb E(|D[\mathcal P,\mathcal B]|^p)\approx_p\frac1{N^p}\mathbb E\left(\left(\sum_{j\in\mathcal J}|\eta_j|^{2}\right)^{p/2}\right)\leqslant
\frac1{N^p}\mathbb E\left(|\mathcal J|^{p/2}\right)=O(N^{-3p/4}).
\end{equation*}
We will see that one such argument can be also used to deduce estimates on the error in numerical integration.
In particular, we will study two types of problems. In the first case 
we will focus on the worst case numerical integration error,
which determines how bad the error of a given fixed quadrature rule can be 
when applied to all integrands whose norm has an upper bound. The function space that we will consider
for this type of problem is a space of potentials: We will say that 
$f\in\mathbb H^\Phi_p(\mathcal M)$, for $1\leqslant p\leqslant \infty$
and a suitable kernel $\Phi(x,y)$ defined on $\mathcal M\times\mathcal M$
(see \S \ref{sec 1} for the precise definition), if there is a $g\in\mathbb L^p(\mathcal M)$
such that $$f(x)=\int_{\mathcal M}\Phi(x,y)g(y)dy,$$ and its norm 
is $\|f\|_{\mathbb H^{\Phi}_p(\mathcal M)}=\inf_g\|g\|_{\mathbb L^p(\mathcal M)}$,
where the infimum is taken as $g$ varies among all functions satisfying 
the previous identity.
Observe that when $\mathcal M$ is the Euclidean space $\mathbb R^d$
and $\Phi(x,y)=|x-y|^{\alpha-d}$, $0<\alpha<d$, is the Riesz kernel,
then for $1<p<\infty$ the potential space $\mathbb H^{\Phi}_p(\mathcal M)$
coincides with the homogeneous fractional Sobolev space $\dot{\mathbb H}^\alpha_p(\mathbb R^d)$.
The (non-homogeneous) fractional Sobolev space ${\mathbb H}^\alpha_p(\mathbb R^d)$
can be obtained similarly, via the Bessel kernel. 
Also, when $\mathcal M$ is a compact Riemannian manifold, the Sobolev space  $\mathbb H^\alpha_p(\mathcal M)$ 
can be defined
as a potential space via the Bessel kernel, see Example \ref{example} here or \cite{BCCGST} for
details on this. We will show here the following
\begin{theoremintro} 
Let $\mathcal{M}$ be a metric measure space with the
property that there exist $d$ and $c$ such that for every $y\in\mathcal{M}$
and $r>0$,
\[
\left\vert \left\{  x\in\mathcal{M}:\left\vert x-y\right\vert \leqslant
r\right\}  \right\vert \leqslant cr^{d}.
\]
Assume also that $\mathcal{M}$ can be decomposed into a finite disjoint union
of sets in the form $\mathcal{M=X}_{1}\cup\cdots\cup\mathcal{X}_{N}$, with
$\omega_{j}=\left\vert \mathcal{X}_{j}\right\vert \approx N^{-1}$ and
$\delta_{j}=\operatorname*{diam}\left(  \mathcal{X}_{j}\right)  \approx
N^{-1/d}$. Assume that for some $0<\alpha<d$,
\[
\left\vert \Phi\left(  x,y\right)  \right\vert \leqslant c\left\vert
x-y\right\vert ^{\alpha-d} \quad \text{for every $x$ and $y$,}
\]
\[
\left\vert \Phi\left(  x,y\right)  -\Phi\left(  z,y\right)  \right\vert
\leqslant c\left\vert x-z\right\vert  \left\vert x-y\right\vert
^{\alpha-d-1} \quad \text{if $\left\vert x-y\right\vert \geqslant2\left\vert x-z\right\vert $.}
\]
Finally,
assume that $1<p\leqslant+\infty$, $1/p+1/q=1$ and $d/p<\alpha<d$. Then
\[
\left\{  {\displaystyle\int_{\mathbf{X}}} \sup_{\|f\|_{\mathbb H^\Phi_p}\leqslant 1}\left\vert \mathcal{E}%
_{\mathbf{x},\mathbf{\omega}}(f)\right\vert ^{q} d\mathbf{x}\right\}  ^{1/q}
\leqslant\left\{
\begin{array}
[c]{ll}%
cN^{-\alpha/d} & \text{if $\alpha<d/2+1$},\\
cN^{-1/2-1/d}\left(  \log N\right)  ^{1/2} & \text{if $\alpha
=d/2+1$},\\
cN^{-1/2-1/d} & \text{if $\alpha>d/2+1$}.
\end{array}
\right.
\]
\end{theoremintro}
The first observation is that the Bessel kernel on a compact 
Riemannian manifold satisfies the hypotheses required by this theorem.
In fact, the particular case given by the case $\mathcal M$
 a compact Riemannian manifold, $\Phi$ the Bessel kernel, $\alpha<d/2+1$,
and $p=2$ had been proved in \cite[Theorem 2.7]{BCCGST}
(see also \cite[Theorem 24]{BSSW} for the case of the sphere).

We also show that under rather natural hypotheses on the space 
$\mathcal M$ and on the kernel $\Phi$, the above estimates
from above hold from below as well.
\begin{theoremintro}
Let $\mathcal{M}$ be a metric measure space with the
property that there exist $H,K$ and $d$ such that for every $y\in\mathcal{M} $
and $0<r<r_{0}$,
\[
Hr^{d}\leqslant\left\vert \left\{  x\in\mathcal{M}:\left\vert x-y\right\vert
\leqslant r\right\}  \right\vert \leqslant Kr^{d}.
\]
Assume also that $\mathcal{M}$ can be decomposed into a finite disjoint union
of sets in the form $\mathcal{M=X}_{1}\cup\cdots\cup\mathcal{X}_{N}$, with
$\omega_{j}=\left\vert \mathcal{X}_{j}\right\vert \approx N^{-1}$ and
$\delta_{j}=\operatorname*{diam}\left(  \mathcal{X}_{j}\right)  \approx
N^{-1/d} $. Suppose that there exists $0<\alpha<d$, such
that for any $j=1,\ldots,N$ and any $z\in\mathcal{X}_{j}$, and for any $y$
such that $\operatorname{dist}\left(  y,\mathcal{X}_{j}\right)  \geqslant
2\delta_{j}$,
\[
\int_{\mathcal{X}_{j}}\left\vert \Phi\left(  x,y\right)  -\Phi\left(
z,y\right)  \right\vert dx\geqslant cN^{-1-1/d}\left(
\operatorname{dist}\left(  y,\mathcal{X}_{j}\right)  \right)  ^{\alpha
-d-1}.
\]
Suppose also that for any $y\in M$, the function $x\mapsto\Phi\left(
x,y\right)  $ is continuous in $x\neq y$. Finally, assume that $1<p\leqslant
+\infty$, $1/p+1/q=1$ and $d/p<\alpha<d$. Then
\[
\left\{  {\int_{\mathbf{X}}}\sup_{\|f\|_{\mathbb H^\Phi_p}\leqslant 1}\left\vert \mathcal{E}_{\mathbf{x},\mathbf{\omega
}}(f)\right\vert ^{q}d\mathbf{x}\right\}  ^{1/q}\geqslant\left\{
\begin{array}
[c]{ll}%
cN^{-\alpha/d} & \text{if $\alpha<d/2+1$},\\
cN^{-1/2-1/d}\left(  \log N\right)  ^{1/2} & \text{if $\alpha
=d/2+1$},\\
cN^{-1/2-1/d} & \text{if $\alpha>d/2+1$}.
\end{array}
\right.
\]
\end{theoremintro}

Once again, it should be observed that the Bessel kernel on a compact Riemannian manifold satisfies 
the hypotheses in the above theorem, and that the particular case given by the case where $\mathcal M$ is
the sphere, $\Phi$ the Bessel kernel, and $p=2$ had been proved in 
\cite[Theorems 24 and 25]{BSSW}.

In order to understand the two above results it could be interesting to recall the
following result \cite[Theorem 2.16]{BCCGST}:  Let $\mathcal M$
be a compact Riemannian manifold. For every
$1\leqslant p\leqslant \infty$ and $\alpha>d/p$ there exists $c>0$
such that for every distribution of points $\mathbf x=\{x_j\}_{j=1}^N$ and
weights $\mathbf \omega=\{\omega_j\}_{j=1}^N$, one has 
\[
\sup_{\|f\|_{\mathbb H^\alpha_p}\leqslant 1}
\left|\mathcal E_{\mathbf x,\mathbf \omega}(f)\right|\geqslant cN^{-\alpha/d}.
\]

In other words, the worst case error for any quadrature rule cannot have 
a better decay than $N^{-\alpha/d}$. Thus, Theorem I says that a random choice
of points $x_j\in\mathcal X_j$, $j=1,\ldots,N$, gives the best possible
decay for the worst case error in $\mathbb H^\alpha_p(\mathcal M)$ when $d/2<\alpha<d/2+1$,
while Theorem II says that when $\alpha\geqslant d/2+1$ the stratification strategy does not lead,
on average, to quadrature rules with the desired decay $N^{-\alpha/d}$ of the worst case error 
in $\mathbb H^\alpha_p(\mathcal M)$.

By the above mentioned result in \cite{G-L} on the possibility
of partitioning $\mathcal M$ into $N$ regions of equal measure and small diameter, 
under the hypotheses on $\mathcal M$ contained in Theorem II 
(see  \S \ref{sec GL} below), the above Theorems I and II apply with equal weights, that is $\omega_j=|\mathcal M|/N$
for all $j=1,\ldots,N$. Point configurations on the sphere that give the best possible decay $N^{-\alpha/d}$
for the worst case error in the equal weight case have been called Quasi Monte Carlo (QMC) designs
in \cite{BSSW}.

We would like to emphasize that all the cited results in \cite{BCCGST} and 
\cite{BSSW}  are based on Hilbert space techniques ($p=2$), while we were able to obtain 
$\mathbb L^p$ integrability results thanks to the Marcinkiewicz--Zygmund inequality. 
Moreover, we work in the more general setting of metric measure spaces, and some
of the results are new even in the particular case of compact Riemannian manifolds.

So far, we have considered the worst case error in numerical integration, that is the error for a whole class of functions. The second type of estimate that will be discussed here concerns the numerical integration
error for a given fixed function. In particular, we will consider functions in the homogeneous 
Haj\l asz--Besov space $\dot{\mathbb B}^\alpha_{p,\infty}(\mathcal M)$, as defined in \cite{KYZ}.
For details on these spaces, see \S \ref{sec 2} below. Here we should mention that 
when $0<\alpha<1$, the spaces $\mathbb H^{\Phi}_{p}(\mathcal M)$
as in Theorem I are embedded into $\dot{\mathbb B}^\alpha_{p,\infty}(\mathcal M)$,
and that when $\mathcal M$ is the Euclidean space $\mathbb R^d$ and $0<\alpha<1$
then the spaces $\dot{\mathbb B}^\alpha_{p,\infty}(\mathcal M)$ coincide with the
classical homogeneous Besov spaces defined via Littlewood--Paley decomposition.
The main result in this context is the following

\begin{theoremintro}
Assume that a metric measure space $\mathcal{M}$ can be
decomposed into a finite disjoint union of sets in the form $\mathcal{M=X}%
_{1}\cup\cdots\cup\mathcal{X}_{N}$, with measure $0<\left\vert \mathcal{X}%
_{j}\right\vert =\omega_{j}\approx N^{-1}$ and $0<\operatorname*{diam}\left(
\mathcal{X}_{j}\right)  \approx{N^{-1/d}}$.  Then for every  $1\leqslant p\leqslant2$ there is 
a constant $c$ such that 
\begin{equation*}
\left\{  {\int_{\mathbf{X}}}\left\vert \mathcal{E}_{\mathbf{x},\mathbf{\omega
}}\left(  f\right)  \right\vert ^{p}d\mathbf{x}\right\}  ^{1/p}\leqslant
c  N^{ 1/p-1-\alpha/d}   \left\Vert f\right\Vert
_{\dot{\mathbb{B}}_{p,\infty}^{\alpha}\left(  \mathcal{M}\right)  }%
\end{equation*}
and for every $2\leqslant p<+\infty$ there is a constant $c$ such that 
\begin{equation*}
\left\{  {\int_{\mathbf{X}}}\left\vert \mathcal{E}_{\mathbf{x},\mathbf{\omega
}}\left(  f\right)  \right\vert ^{p}d\mathbf{x}\right\}  ^{1/p}\leqslant
cN^{-1/2-\alpha/d}  \left\Vert f\right\Vert _{\dot{\mathbb{B}}_{p,\infty}^{\alpha}\left(
\mathcal{M}\right)  }.%
\end{equation*}
\end{theoremintro}
Of course, here a random choice of points $x_j\in\mathcal X_j$, $j=1,\ldots,N$ gives better estimates than those obtained in
Theorem I (and better than $N^{-\alpha/d}$). This is natural, since in this case we are looking for
point distributions which give a small error for a given integrand $f$, whereas in the situation described
by Theorem I we were looking for point distributions which give a small error for all integrands in
our space at the same time.

Theorem III and its sharpness will be discussed in \S\ref{numerical 2} below.
\medskip

We believe that our effort in the search
for the minimal properties that guarantee the validity of certain
results in discrepancy theory and numerical integration can be of some help
towards a deeper understanding of these results, even in the classical cases.
In fact, to the best of our knowledge, the above Theorems I, II and III are new even when applied to a general
compact Riemannian manifold.
 There has recently been some interest in this type
of problems in spaces as general as those considered here, or even more.
See for example \cite{MM}, \cite{RS}, \cite{Skr} for discrepancy and numerical integration on metric spaces and \cite{Str1}, \cite{Str2}  for analysis on fractals.
\medskip

The plan of the paper is the following. In \S \ref{sec 1} and \S \ref{sec 2} we introduce
the appropriate Sobolev-type spaces, and we recall a few results 
on how these spaces relate to each other. These matters are not completely new, but can be of
some help for the unfamiliar reader. In \S \ref{sec 3} we introduce in some detail the Marcinkiewicz--Zygmund 
inequality. In \S \ref{sec GL} we recall the above mentioned result in \cite{G-L} concerning the possibility
of partitioning a metric measure space into regions of equal measure and small diameter.
In \S \ref{numerical 1} and \S\ref{numerical 2} we give all the details on our results on numerical 
integration, with examples. Finally, \S \ref{discr} contains
our results on the $L^p$ (and $L^\infty$) discrepancy that generalize \eqref{basic} and \eqref{discs}.

\section{\label{sec 1}Sobolev spaces and potentials on measure spaces}

Our estimates on the worst case error described in Theorems I and II
require a definition of Sobolev spaces via
potentials. For a classical approach to potentials, see, for example, \cite{S}.

\begin{definition}
\label{def:potentials} Let $\mathcal{M}$ be a measure space, let $1\leqslant
p,q\leqslant+\infty$ with $1/p+1/q=1,$ and let $\Phi\left(  x,y\right)  $ be a
measurable kernel on $\mathcal{M}\times\mathcal{M}$. Assume that for every $x
$,
\[%
\begin{array}
[c]{rl}%
{{%
{\displaystyle\int_{\mathcal{M}}}
}\left\vert \Phi\left(  x,y\right)  \right\vert ^{q}dy<+\infty} & \text{if
$q<+\infty$},\\
{\underset{y\in\mathcal{M}}{\operatorname*{ess}\sup}\left\{  \left\vert
\Phi\left(  x,y\right)  \right\vert \right\}  <+\infty} & \text{if $q=+\infty
$}.
\end{array}
\]
Then every function $g(x)$ in $\mathbb{L}^{p}\left(  \mathcal{M}\right)  $ has
a pointwise well defined potential
\[
f\left(  x\right)  ={\int_{\mathcal{M}}}\Phi\left(  x,y\right)  g\left(
y\right)  dy.
\]
The space $\mathbb{H}_{p}^{\Phi}\left(  \mathcal{M}\right)  $ is the space of
all potentials of functions in $\mathbb{L}^{p}\left(  \mathcal{M}\right)  $,
with norm
\[
\left\Vert f\right\Vert _{\mathbb{H}_{p}^{\Phi}\left(  \mathcal{M}\right)
}=\inf_{g}\left\{  {\int_{\mathcal{M}}}\left\vert g\left(  x\right)
\right\vert ^{p}dx\right\}  ^{1/p}.
\]
The infimum is taken over all $g(x)$ which give the potential $f(x)$.
\end{definition}

Observe that this definition does not even require a metric. Potentials can
also be defined under weaker assumptions on the kernel, but the above
assumptions guarantee that these potentials are defined pointwise everywhere,
and this will be necessary in the sequel. In particular, when $\mathcal{M}$ is
the Euclidean space $\mathbb{R}^{d}$ and $\Phi\left(  x,y\right)  =\left\vert
x-y\right\vert ^{\alpha-d}$ with $0<\alpha<d$ is the Riesz kernel, then
$\mathbb{H}_{p}^{\Phi}\left(  \mathcal{M}\right)  $ is the homogeneous fractional Sobolev
space $\dot{\mathbb{H}}_{p}^{\alpha}\left(  \mathbb{R}^{d}\right)  $. However, the
cases $p=1$ and $p=+\infty$ require some extra care. For interesting examples
of generalized potential spaces, see, for example, \cite{HZ}.

\section{\label{sec 2}Besov and Triebel--Lizorkin spaces on metric measure
spaces}

Our estimates in Theorem III require a definition of Sobolev spaces, more appropriately 
Besov or Triebel--Lizorkin spaces, via
upper gradients. Let $\mathcal{M}$ be a metric measure space, that is, a
metric space equipped with a positive Borel measure. With a small abuse of
notation we denote by $\left\vert \mathcal{X}\right\vert $ the measure of a
measurable set $\mathcal{X}$ and by $\left\vert x-y\right\vert $ the distance
between two points $x$ and $y$. We will often denote with $B(x,r)$ the open balls 
$\{y\in\mathcal M:|x-y|<r\}$ with center 
$x$ and radius $r$. Simple examples are Riemannian manifolds, or
non necessarily smooth surfaces in a Euclidean space with the inherited
measure and distance. In \cite{H} Haj\l asz has given a purely metric
definition of a Sobolev space: A measurable function $f(x)$ is in the Sobolev
space $\mathbb{W}_{p}^{1}\left(  \mathcal{M}\right)  $, $1\leqslant
p\leqslant+\infty$, if there exists a non-negative function $g\left(
x\right)  $ in $\mathbb{L}^{p}\left(  \mathcal{M}\right)  $ such that for
almost every $x,y\in\mathcal{M}$,
\[
\left\vert f(x)-f(y)\right\vert \leqslant\left\vert x-y\right\vert \left(
g(x)+g(y)\right)  .
\]

For example, it is proved in \cite{H} that in Euclidean spaces one can choose
as an upper gradient $g(x)$ a suitable multiple of the Hardy--Littlewood maximal function of the
gradient $\nabla f\left(  x\right)  $.

The following is a natural generalization of upper gradient and associated
Sobolev space.

\begin{definition}
\label{def:hajlasz-Sobolev} Let $\mathcal{M}$ be a metric measure space and
$\varphi\left(  t\right)  $ a non-negative increasing function in
$t\geqslant0$. A measurable non-negative function $g\left(  x\right)  $ is a
$\varphi$-gradient of a measurable function $f\left(  x\right)  $ if there
exists a set $\mathcal{N} $ with measure zero such that for all $x$ and $y$ in
$\mathcal{M\setminus N}$,
\[
\left\vert f\left(  x\right)  -f\left(  y\right)  \right\vert \leqslant
\varphi\left(  \left\vert x-y\right\vert \right)  \left(  g\left(  x\right)
+g\left(  y\right)  \right)  .
\]

\end{definition}

\begin{definition}
A measurable function $f(x)$ is in the Haj\l asz--Sobolev space $\dot{\mathbb{M}}_{p}^{\varphi
}\left(  \mathcal{M}\right)  ,$ $0<p\leqslant+\infty$, if it has a $\varphi
$-gradient in $L^{p}\left(  \mathcal{M}\right)  $. We set
\[
\left\Vert f\right\Vert _{\dot{\mathbb{M}}_{p}^{\varphi}\left(  \mathcal{M}\right)
} =\inf\left\Vert g\right\Vert _{L^{p}\left(  \mathcal{M}\right)  }.
\]
The infimum is taken over all $\varphi$-gradients $g\left(  x\right)  $ of
$f\left(  x\right)  $.
\end{definition}

In \cite{H} Haj\l asz proved that when $\varphi(t)=t$, $1<p\le+\infty$ and ${\mathcal M}=\mathbb R^d$,
then this space coincides with the classical homogeneous Sobolev space $\dot{\mathbb H}^{1}_p\left(\mathbb R^d\right)$.
The above definition has been extended by Koskela, Yang and Zhou \cite{KYZ}
who have defined Besov and Triebel--Lizorkin spaces on a general metric
measure space. In particular they proved that, when $\varphi(t)=t^\alpha$ with $0<\alpha<1$ 
and ${\mathcal M}=\mathbb R^d$, then the space 
$\dot{\mathbb{M}}_{p}^{\varphi}\left(  \mathcal{M}\right)$ 
is larger than the classical fractional Sobolev space, and it coincides with the homogeneous Triebel--Lizorkin space
$\dot{\mathbb F}^{\alpha}_{p,\infty}\left(\mathbb R^d\right)$.

What follows is in the spirit of the definitions of Besov and Triebel--Lizorkin spaces in \cite{KYZ}. In order to
define these spaces one needs to introduce families of gradients localized at
different scales.

\begin{definition}
\label{def:succ phi-gradient} Let $\mathcal{M}$ be a metric measure space and
$\varphi\left(  t\right)  $ a non-negative increasing function in
$t\geqslant0$. Let $n_{0}=\log_{2}\left(  \operatorname*{diam}\left(
\mathcal{M}\right)  \right)  $, possibly infinity. A sequence of non-negative
measurable functions $\left\{  g_{n}\left(  x\right)  \right\}  _{-n_{0}%
}^{+\infty}$ is a $\varphi$-gradient for a measurable function $f\left(
x\right)  $ if there exists a set $\mathcal{N}$ with measure zero such that
for all $x$ and $y$ in $\mathcal{M\setminus N}$ with $\left\vert
x-y\right\vert \leqslant2^{-n}$,
\[
\left\vert f\left(  x\right)  -f\left(  y\right)  \right\vert \leqslant
\varphi\left(  2^{-n}\right)  \left(  g_{n}\left(  x\right)  +g_{n}\left(
y\right)  \right)  .
\]

\end{definition}

\begin{definition}
A measurable function $f(x)$ is in the Haj\l asz--Triebel--Lizorkin space $\dot{\mathbb{F}}%
_{p,q}^{\varphi}\left(  \mathcal{M}\right)  ,$ $0<p<+\infty$ and
$0<q\leqslant+\infty$, if $f(x)$ has a $\varphi$-gradient $\left\{
g_{n}\left(  x\right)  \right\}  $ with
\[%
\begin{array}
[c]{rl}%
{\left\Vert \left\{  {%
{\displaystyle\sum\limits_{n=-n_{0}}^{+\infty}}
}\left\vert g_{n}\left(  x\right)  \right\vert ^{q}\right\}  ^{1/q}\right\Vert
_{\mathbb{L}^{p}\left(  \mathcal{M}\right)  }<+\infty} & \text{if
\ $0<q<+\infty$},\\
{\left\Vert \sup_{n\geqslant-n_{0}}\left\vert g_{n}\left(  x\right)
\right\vert \right\Vert _{\mathbb{L}^{p}\left(  \mathcal{M}\right)  }<+\infty}
& \text{if \ $q=+\infty$}.
\end{array}
\]
The infimum of the above expression taken over all $\varphi$-gradients defines
the semi-norm $\left\Vert f\right\Vert _{\dot{\mathbb{F}}_{p,q}^{\varphi}\left(
\mathcal{M}\right)  }$.
\end{definition}

\begin{definition}
A measurable function $f(x)$ is in the Haj\l asz--Besov space $\dot{\mathbb{B}}_{p,q}^{\varphi
}\left(  \mathcal{M}\right)  $, $0<p\leqslant+\infty$ and $0<q\leqslant
+\infty$, if $f(x)$ has a $\varphi$-gradient $\left\{  g_{n}\left(  x\right)
\right\}  $ with
\[%
\begin{array}
[c]{rl}%
{\left\{  {%
{\displaystyle\sum\limits_{n=-n_{0}}^{+\infty}}
}\left\Vert g_{n}\left(  x\right)  \right\Vert _{\mathbb{L}^{p}\left(
\mathcal{M}\right)  }^{q}\right\}  ^{1/q}<+\infty} & \text{if \ $0<q<+\infty$%
},\\
{\sup_{n\geqslant-n_{0}}\left\Vert g_{n}\left(  x\right)  \right\Vert
_{\mathbb{L}^{p}\left(  \mathcal{M}\right)  }<+\infty} & \text{if
\ $q=+\infty$}.
\end{array}
\]
The infimum of the above expression taken over all $\varphi$-gradients defines
the semi-norm $\left\Vert f\right\Vert _{\dot{\mathbb{B}}_{p,q}^{\varphi}\left(
\mathcal{M}\right)  }$.
\end{definition}

Observe that the above spaces are homogeneous, and constant functions have
semi-norms equal to zero. Also observe that when $q=+\infty$ and $\varphi(t)$ is doubling,
that is $\varphi(2t)\le c\varphi(t)$, then the space
$\dot{\mathbb{F}}_{p,\infty}^{\varphi}\left(  \mathcal{M}\right)  $ coincides with
the previously defined Haj\l asz--Sobolev space $\dot{\mathbb{M}}_{p}^{\varphi}\left(  \mathcal{M}\right)  $. It
suffices to define $g\left(  x\right)  =\sup\left\{  g_{n}\left(  x\right)
\right\}  $. In particular, the straightforward generalization of a Haj\l asz--Sobolev space is
a Haj\l asz--Triebel--Lizorkin space.

When $\varphi\left(  t\right)  =t^{\alpha}$ the above definition is nothing
but the definition of Haj\l asz--Besov and Haj\l asz--Triebel--Lizorkin given in \cite{KYZ}. To be
precise, the definition of $\varphi$-gradient in \cite{KYZ} requires
\[
\left\vert f\left(  x\right)  -f\left(  y\right)  \right\vert \leqslant
2^{-\alpha n}\left(  g_{n}\left(  x\right)  +g_{n}\left(  y\right)  \right)
\]
only for $x$ and $y$ with $2^{-n-1}\leqslant\left\vert x-y\right\vert
\leqslant2^{-n}$. On the other hand, if one defines
\[
G_{n}\left(  x\right)  ={\sum_{k=0}^{+\infty}}2^{-\alpha k}g_{n+k}\left(
x\right)  ,
\]
then
\[
\left\vert f\left(  x\right)  -f\left(  y\right)  \right\vert \leqslant
2^{-\alpha n}\left(  G_{n}\left(  x\right)  +G_{n}\left(  y\right)  \right)
\]
for $x$ and $y$ with $\left\vert x-y\right\vert \leqslant2^{-n}$, and the
seminorms defined via $\left\{  g_{n}\left(  x\right)  \right\}  $ and
$\left\{  G_{n}\left(  x\right)  \right\}  $ are equivalent. In the same
paper, it is proved that when $\mathcal{M}$ is the Euclidean space
$\mathbb{R}^{d}$ and $\varphi\left(  t\right)  =t^{\alpha}$ with $0<\alpha<1$,
then the spaces $\dot{\mathbb{B}}_{p,q}^{\varphi}\left(  \mathcal{M}\right)  $ and
$\dot{\mathbb{F}}_{p,q}^{\varphi}\left(  \mathcal{M}\right)  $ coincide with the
classical Besov and Triebel--Lizorkin spaces defined via a Littlewood--Paley
decomposition. See, for example, \cite{BL} for the relevant definitions.

The lemma below gives an example of a function in the spaces $\dot{\mathbb{B}}%
_{p,q}^{\varphi}\left(  \mathcal{M}\right)  $ and $\dot{\mathbb{F}}_{p,q}^{\varphi
}\left(  \mathcal{M}\right)  $.

\begin{definition}
For every subset $\mathcal{B}$ in $\mathcal{M}$, define
\[
\psi_{\mathcal{B}}\left(  t\right)  =\left\vert \left\{  x\in\mathcal{B}%
:\operatorname*{dist}\left\{  x,\mathcal{M}\setminus\mathcal{B}\right\}
\leqslant t\right\}  \right\vert +\left\vert \left\{  x\in\mathcal{M}%
\setminus\mathcal{B}:\operatorname*{dist}\left\{  x,\mathcal{B}\right\}
\leqslant t\right\}  \right\vert .
\]

\end{definition}

For example, if $\mathcal{M}$ is a $d$-dimensional Riemannian manifold and
$\mathcal{B}$ is a bounded open set with regular boundary, then $\psi
_{\mathcal{B}}\left(  t\right)  \approx t$, while if $\psi_{\mathcal{B}%
}\left(  t\right)  \approx t^{\beta}$, then the boundary has Minkowski fractal
dimension $d-\beta$.

\begin{proposition}
\label{ex:upper_gradient} Let $\mathcal{B}$ be an arbitrary subset of
$\mathcal{M}$. Then
\begin{align*}
\left\Vert \chi_{\mathcal{B}}\right\Vert _{\dot{\mathbb{F}}_{p,\infty}^{\varphi
}\left(  \mathcal{M}\right)  }  &  \leqslant\left\{  \sum_{n=-n_{0}}^{+\infty}
\varphi\left(  2^{-n}\right)  ^{-p}\psi_{\mathcal{B}}\left(  2^{-n}\right)
\right\}  ^{1/p},\\
\left\Vert \chi_{\mathcal{B}}\right\Vert _{\dot{\mathbb{B}}_{p,\infty}^{\varphi}
\left(  \mathcal{M}\right)  }  &  \leqslant\sup_{n\geqslant-n_{0}} \left\{
\varphi\left(  2^{-n}\right)  ^{-1}\psi_{\mathcal{B}}\left(  2^{-n}\right)
^{1/p}\right\}  .
\end{align*}
In particular, if $\varphi\left(  t\right)  =t^{\alpha}$ and $\psi
_{\mathcal{B}}\left(  t\right)  \leqslant ct^{\beta}$, then $\chi
_{\mathcal{B}}\in\dot{\mathbb{F}}_{p,\infty}^{\varphi}\left(  \mathcal{M}\right)  $
for $p\alpha<\beta$ and $\chi_{\mathcal{B}}\in\dot{\mathbb{B}}_{p,\infty}^{\varphi
}\left(  \mathcal{M}\right)  $ for $p\alpha\leqslant\beta$.
\end{proposition}

\begin{proof}
It suffices to observe that a $\varphi$-gradient for the characteristic
function $\chi_{\mathcal{B}}\left( x\right) $ is given by
\[
g_{n}(x)=\left\{
\begin{array}
[c]{ll}%
\varphi\left( 2^{-n}\right) ^{-1} & \text{if $x\in\mathcal{B}$ and
$\operatorname*{dist}\left\{ x,\mathcal{M}\setminus\mathcal{B}\right\} \leqslant2^{-n}$%
},\\
0 & \text{otherwise}.
\end{array}
\right.
\]

Of course, there are other possible choices for the $\varphi$-gradient of
$\chi_{\mathcal B}\left(  x\right)  $, for example,
\[
g_{n}(x)=\left\{
\begin{array}
[c]{ll}%
\varphi\left(  2^{-n}\right)  ^{-1} & \text{if $x\in\mathcal{M}\setminus
\mathcal{B}$ and $\operatorname*{dist}\left\{  x,\mathcal{B}\right\}
\leqslant2^{-n} $},\\
0 & \text{otherwise}.
\end{array}
\right.
\]
\end{proof}

The following is an immediate consequence of the definitions.

\begin{proposition}
(i) If $q_{1}\leqslant q_{2}$ and $\varphi_{1}\left(  t\right)  \leqslant
\varphi_{2}\left(  t\right)  $, then
\[
\dot{\mathbb{B}}_{p,q_{1}}^{\varphi_{1}}\left(  \mathcal{M}\right)  \subseteq
\dot{\mathbb{B}}_{p,q_{2}}^{\varphi_{2}}\left(  \mathcal{M}\right)  \quad
\mbox{and}\quad\dot{\mathbb{F}}_{p,q_{1}}^{\varphi_{1}}\left(  \mathcal{M}\right)
\subseteq\dot{\mathbb{F}}_{p,q_{2}}^{\varphi_{2}}\left(  \mathcal{M}\right)  .
\]
(ii) For every $\varphi\left(  t\right)  $ and $0<p\leqslant+\infty$,
\[
\dot{\mathbb{B}}_{p,p}^{\varphi}\left(  \mathcal{M}\right)  ={\dot{\mathbb{F}}}%
_{p,p}^{\varphi}\left(  \mathcal{M}\right)  \quad\mbox{and}\quad
\dot{\mathbb{F}}_{p,\infty}^{\varphi}\left(  \mathcal{M}\right)  \subseteq
\dot{\mathbb{B}}_{p,\infty}^{\varphi}\left(  \mathcal{M}\right)  .
\]
In particular, for fixed $p$ and $\varphi\left(  t\right)  $, the largest
space in the scale of Haj\l asz--Besov and Haj\l asz--Triebel--Lizorkin spaces is $\dot{\mathbb{B}}%
_{p,\infty}^{\varphi}\left(  \mathcal{M}\right)  $.
\end{proposition}

In the Euclidean spaces it is well known that the homegeneous potential spaces $\dot{\mathbb{H%
}}_{p}^{\alpha }\left( \mathbb{R}^{d}\right) $ defined via the Riesz
potentials coincide with the homogeneous Triebel--Lizorkin spaces $\dot{\mathbb{F}}%
_{p,2}^{\alpha }\left( \mathbb{R}^{d}\right) $ defined via the Littlewood--Paley decomposition
 \cite{BL,FJW}. We do not know under which assumptions on $\Phi
\left( x,y\right) $ and $\varphi \left( t\right) $ and $\mathcal{M}$ the
equality $\mathbb{H}_{p}^{\Phi }\left( \mathcal{M}\right) =\dot{\mathbb{F}}%
_{p,2}^{\varphi }\left( \mathcal{M}\right) $ holds. Anyhow, the following proposition guarantees 
a weaker embedding.

\begin{proposition} Assume that $\psi \left( t\right) $
is an increasing function on $0\leqslant t<+\infty $ with $\psi \left(
0\right) =0$, and define for $\varepsilon>0$ %
\begin{equation*}
\varphi \left( t\right) =\sum_{k=0}^{+\infty }\psi \left( 2^{2-k}t\right)
+\sum_{k=0}^{+\infty }2^{-(k+1)\varepsilon}\psi \left( 2^{k+2}t\right) .
\end{equation*}
Also assume that $\Phi \left( x,y\right) $ is a kernel on %
$\mathcal{M}\times \mathcal{M}$, with the property that for some $%
C>0$,%
\begin{align*}
\left\vert \Phi \left( x,y\right) \right\vert & \leqslant C\frac{\psi \left( \left\vert
x-y\right\vert \right)}{ \left\vert B\left( x,6\left\vert x-y\right\vert
\right) \right\vert} \quad \text{ for every $x,y\in \mathcal M$,}\\
\left\vert \Phi \left( x,y\right) -\Phi \left( z,y\right) \right\vert &\leqslant
C\left(\frac{\left\vert x-z\right\vert}{ \left\vert x-y\right\vert}\right) ^{\varepsilon}\frac{\psi \left(
\left\vert x-y\right\vert \right)}{ \left\vert B\left( x,6\left\vert
x-y\right\vert \right) \right\vert }\quad\text{ when }|x-y|\geqslant2|x-z|.
\end{align*}
Define the potential %
\begin{equation*}
f\left( x\right) =\int_{\mathcal{M}}\Phi \left( x,y\right) g\left( y\right)
dy.
\end{equation*}
Finally, define the Hardy--Littlewood maximal operator %
\begin{equation*}
Mg\left( x\right) =\sup_{r>0}\left\{ \left\vert B\left( x,3r\right)
\right\vert ^{-1}\int_{\left\{ \left\vert x-y\right\vert \leqslant r\right\}
}\left\vert g\left( y\right) \right\vert dy\right\} .
\end{equation*}
Then, %
\begin{equation*}
\left\vert f\left( x\right) -f\left( z\right) \right\vert \leqslant 2C\varphi
\left( \left\vert x-z\right\vert \right) \left( Mg\left( x\right) +Mg\left(
z\right) \right) .
\end{equation*}
\end{proposition}

\begin{proof} 
By the hypotheses on the kernel,
\begin{align*}
\left\vert f\left( x\right) -f\left( z\right) \right\vert &\leqslant \int_{%
\mathcal{M}}\left\vert \Phi \left( x,y\right) -\Phi \left( z,y\right)
\right\vert \left\vert g\left( y\right) \right\vert dy \\
&\leqslant C\int_{\left\{ \left\vert x-y\right\vert \leqslant 4\left\vert
x-z\right\vert \right\} }\psi \left( \left\vert x-y\right\vert \right)
\left\vert B\left( x,6\left\vert x-y\right\vert \right) \right\vert
^{-1}\left\vert g\left( y\right) \right\vert dy \\
&+C\int_{\left\{ \left\vert z-y\right\vert \leqslant 4\left\vert x-z\right\vert
\right\} }\psi \left( \left\vert z-y\right\vert \right) \left\vert B\left(
z,6\left\vert z-y\right\vert \right) \right\vert ^{-1}\left\vert g\left(
y\right) \right\vert dy \\
&+C\int_{\left\{ \left\vert x-y\right\vert \geqslant 2\left\vert x-z\right\vert
\right\} }\left\vert x-z\right\vert^{\varepsilon} \left\vert x-y\right\vert ^{-\varepsilon}\psi
\left( \left\vert x-y\right\vert \right) \left\vert B\left( x,6\left\vert
x-y\right\vert \right) \right\vert ^{-1}\left\vert g\left( y\right)
\right\vert dy.
\end{align*}
The dyadic decomposition 
\[
\left\{ \left\vert x-y\right\vert \leqslant 4\left\vert
x-z\right\vert \right\} =\bigcup _{k=0}^{+\infty }\left\{ 2^{1-k}\left\vert
x-z\right\vert < \left\vert x-y\right\vert \leqslant 2^{2-k}\left\vert
x-z\right\vert \right\}
\] 
gives 
\begin{align*}
&\int_{\left\{ \left\vert x-y\right\vert \leqslant 4\left\vert x-z\right\vert
\right\} }\psi \left( \left\vert x-y\right\vert \right) \left\vert B\left(
x,6\left\vert x-y\right\vert \right) \right\vert ^{-1}\left\vert g\left(
y\right) \right\vert dy \\
\leqslant &\sum_{k=0}^{+\infty }\psi \left( 2^{2-k}\left\vert x-z\right\vert
\right) \left\vert B\left( x,6\cdot 2^{1-k}\left\vert x-z\right\vert \right)
\right\vert ^{-1}\int_{\left\{ \left\vert x-y\right\vert \leqslant
2^{2-k}\left\vert x-z\right\vert \right\} }\left\vert g\left( y\right)
\right\vert dy \\
\leqslant& \sum_{k=0}^{+\infty }\psi \left( 2^{2-k}\left\vert x-z\right\vert
\right) \sup_{r>0}\left\{ \left\vert B\left( x,3r\right) \right\vert
^{-1}\int_{\left\{ \left\vert x-y\right\vert \leqslant r\right\} }\left\vert
g\left( y\right) \right\vert dy\right\} \\
\leqslant & \varphi \left( \left\vert x-z\right\vert \right) Mg\left( x\right) .
\end{align*}
Similarly, 
\begin{equation*}
\int_{\left\{ \left\vert z-y\right\vert \leqslant 4\left\vert x-z\right\vert
\right\} }\psi \left( \left\vert z-y\right\vert \right) \left\vert B\left(
z,6\left\vert z-y\right\vert \right) \right\vert ^{-1}\left\vert g\left(
y\right) \right\vert dy\leqslant \varphi \left( \left\vert x-z\right\vert \right)
Mg\left( z\right) .
\end{equation*}
Finally, the dyadic decomposition 
\[
\left\{ \left\vert x-y\right\vert \geqslant
2\left\vert x-z\right\vert \right\} =\bigcup _{k=0}^{+\infty }\left\{
2^{k+1}\left\vert x-z\right\vert \leqslant \left\vert x-y\right\vert <
2^{k+2}\left\vert x-z\right\vert \right\}
\]
gives 
\begin{align*}
&\int_{\left\{ \left\vert x-y\right\vert \geqslant 2\left\vert x-z\right\vert
\right\} }\left\vert x-z\right\vert^{\varepsilon} \left\vert x-y\right\vert ^{-\varepsilon}\psi
\left( \left\vert x-y\right\vert \right) \left\vert B\left( x,6\left\vert
x-y\right\vert \right) \right\vert ^{-1}\left\vert g\left( y\right)
\right\vert dy \\
\leqslant& \sum_{k=0}^{+\infty }2^{-(k+1)\varepsilon}\psi \left( 2^{k+2}\left\vert
x-z\right\vert \right) \left\vert B\left( x,6\cdot 2^{k+1}\left\vert
x-z\right\vert \right) \right\vert ^{-1}\int_{\left\{ \left\vert
x-y\right\vert \leqslant 2^{k+2}\left\vert x-z\right\vert \right\} }\left\vert
g\left( y\right) \right\vert dy \\
\leqslant& \varphi \left( \left\vert x-z\right\vert \right) Mg\left( x\right) .
\end{align*}
 \end{proof}

\begin{corollary}
\label{ex:potential_vs_Hajlasz}
With the notation of the above proposition, if
$1<p\leqslant +\infty $ then the potential space $\mathbb{H}_{p}^{\Phi
}\left( \mathcal{M}\right) $ can be continuously embedded into $%
\dot{\mathbb{F}}_{p,\infty }^{\varphi }\left( \mathcal{M}\right) $.
\end{corollary}

\begin{proof} It suffices to recall that, due to the extra $3$ in the
definition of $Mg\left( x\right) $, this maximal operator $Mg\left( x\right) 
$ is bounded on $\mathbb{L}^{p}\left( \mathcal{M}\right) $ for all $1<p\leqslant
+\infty $, even when the measure on the metric space is non doubling. See
\cite{NTV}.
 \end{proof}

\begin{example} 
If $\varphi \left( t\right) =t^{\alpha }$ with $0<\alpha
<\varepsilon $, then 
\begin{align*}
\psi \left( t\right) &=\sum_{k=0}^{+\infty }\varphi \left( 2^{2-k}t\right)
+\sum_{k=0}^{+\infty }2^{(-k-1)\varepsilon}\varphi \left( 2^{k+2}t\right) \\
&=\left( 2^{2\alpha }\sum_{k=0}^{+\infty }2^{-k\alpha }+2^{2\alpha
-\varepsilon}\sum_{k=0}^{+\infty }2^{-k\left( \varepsilon-\alpha \right) }\right) t^{\alpha
}=Ct^{\alpha }.
\end{align*}
Thus, if 
\begin{align*}
\left\vert \Phi \left( x,y\right) \right\vert & \leqslant C\frac{\left\vert
x-y\right\vert ^\alpha}{ \left\vert B\left( x,6\left\vert x-y\right\vert
\right) \right\vert}  \quad \text{for every $x,y\in\mathcal{M}$},\\
\left\vert \Phi \left( x,y\right) -\Phi \left( z,y\right) \right\vert &\leqslant
C\left(\frac{\left\vert x-z\right\vert}{ \left\vert x-y\right\vert}\right) ^{\varepsilon}\frac{ 
\left\vert x-y\right\vert ^\alpha}{ \left\vert B\left( x,6\left\vert
x-y\right\vert \right) \right\vert }\quad\text{ when }|x-y|\geqslant2|x-z|.
\end{align*}
and if
$1<p\leqslant +\infty $ then the potential space $\mathbb{H}_{p}^{\Phi
}\left( \mathcal{M}\right) $ can be continuously embedded into $%
\dot{\mathbb{F}}_{p,\infty }^{\varphi }\left( \mathcal{M}\right) $ with $\varphi(t)=t^\alpha$.

\end{example}

\section{\label{sec 3}The Marcinkiewicz--Zygmund inequality}

The main ingredient in what follows is the Marcinkiewicz--Zygmund inequality
for sums of independent random variables.

As is well known, the variance of the sum of independent random variables is
the sum of the variances. For every sequence of independent random variables
$f_{j}$,
\[
\mathbb{E}\left(  \left\vert \sum_{j}\left(  f_{j} -\mathbb{E}\left(
f_{j}\right)  \right)  \right\vert ^{2}\right)  =\sum_{j}\mathbb{E}\left(
\left\vert f_{j} -\mathbb{E}\left(  f_{j}\right)  \right\vert ^{2}\right)  .
\]

In fact, there is a similar result with the second moment replaced by other
moments and with the equality replaced by two inequalities.

\begin{theorem}
[\textbf{Marcinkiewicz--Zygmund}]For every $1\leqslant p<+\infty$, there exist
positive constants $A\left(  p\right)  $ and $B\left(  p\right)  $ such that
for every sequence $\{f_{j}\}$ of independent random variables,
\begin{multline*}
A\left(  p\right)  \left\{  \mathbb{E}\left(  {\displaystyle\sum_{j=1}^{N}}
\left\vert f_{j}-\mathbb{E}\left(  f_{j}\right)  \right\vert ^{2}\right)
^{p/2}\right\}  ^{1/p} \hskip 3cm\\
\leqslant\left\{  \mathbb{E}\left(  \left\vert {\displaystyle\sum_{j=1}^{N}}
\left(  f_{j}-\mathbb{E}\left(  f_{j}\right)  \right)  \right\vert
^{p}\right)  \right\}  ^{1/p}\\
\hskip 3cm \leqslant B\left(  p\right)  \left\{  \mathbb{E}\left(
{\displaystyle\sum_{j=1}^{N}} \left\vert f_{j}-\mathbb{E}\left(  f_{j}\right)
\right\vert ^{2}\right)  ^{p/2}\right\}  ^{1/p}.
\end{multline*}

\end{theorem}

The Marcinkiewicz--Zygmund inequality is a generalization of the classical
inequality of Khintchine for sums of random variables with Rademacher
distribution that take values $\pm1$ with probability $1/2$. For a proof, see
\cite{MZ1} and \cite{MZ2}, or \cite{CT}.

In what follows, special attention will be paid to the constants, and
$A\left(  p\right)  $ and $B\left(  p\right)  $ will denote the best constants
in the Marcinkiewicz--Zygmund inequality. If $\overline{A}\left(  p\right)  $
and $\overline{B}\left(  p\right)  $ are the corresponding best constants for
the Khintchine inequality, then it can be proved that
\[
\tfrac{1}{2}\overline{A}\left(  p\right)  \leqslant A\left(  p\right)
\leqslant\overline{A}\left(  p\right)  \quad\text{and}\quad\overline{B}\left(
p\right)  \leqslant B\left(  p\right)  \leqslant2\overline{B}\left(  p\right)
.
\]
See, for example, \cite{CT}. In particular there is a positive constant $c$
such that $c\leqslant\overline A(p)\leqslant1$, while $\overline B\left(
p\right) =1$ for $1\leqslant p\leqslant2$ and $\overline B(p)=\sqrt{2}\left(
\Gamma\left( ({p+1})/{2}\right) / \sqrt{\pi}\right) ^{1/p}$ for $2\leqslant
p<+\infty$. See \cite{Bu} and \cite{Haa}.

We remark that the Marcinkiewicz--Zygmund inequality can also be extended to
infinite sums of independent random variables.

From now on we will assume that $\mathcal{M}$ is a measure space of finite
measure which can be expressed as a finite union $\mathcal{X}_{1}\cup
\cdots\cup\mathcal{X}_{N}$ of disjoint sets $\mathcal{X}_{1},\cdots
,\mathcal{X}_{N}$, with measure $0<\left\vert \mathcal{X}_{j}\right\vert
=\omega_{j}<+\infty$.

As indicated earlier, we write $\boldsymbol{\omega}=\left(  \omega_{1}%
,\ldots,\omega_{N}\right)  $, $\mathbf{x}=\left(  x_{1},\ldots,x_{N}\right)
$, $\mathbf{X}=\mathcal{X}_{1}\times\cdots\times\mathcal{X}_{N}$ and
\[
d\mathbf{x}=\frac{dx_{1}}{\omega_{1}}\times\cdots\times\frac{dx_{N}}%
{\omega_{N}}.
\]

The error incurred in a quadrature rule with sampling points $\mathbf{x}=\left(
x_{1},\ldots,x_{N}\right)  $ and weights $\boldsymbol{\omega}=\left(
\omega_{1},\ldots,\omega_{N}\right)  $ is the functional
\[
\mathcal{E}_{\mathbf{x,\omega}}\left(  f\right)  ={\sum_{j=1}^{N}}\omega
_{j}f\left(  x_{j}\right)  -\int_{\mathcal{M}}f\left(  x\right)  dx.
\]
The above Marcinkiewicz--Zygmund inequality has an immediate corollary that
allows us to control the norm of this error.

\begin{corollary}
\label{Cor-MZ} Let $1\leqslant p<+\infty$. For every measurable function
$f\left(  x\right)  $ on $\mathcal{M}$,
\begin{multline*}
A\left( p\right)  \left\{  {\int_{\mathbf{X}}}\left(  {\sum_{j=1}^{N}%
}\left\vert \omega_{j}f\left(  x_{j}\right)  -\int_{\mathcal{X}_{j}}f\left(
y_{j}\right)  dy_{j}\right\vert ^{2}\right)  ^{p/2}d\mathbf{x}\right\}
^{1/p}\\
\leqslant\left\{  {\int_{\mathbf{X}}}\left\vert \mathcal{E}_{\mathbf{x,\omega
}}\left(  f\right)  \right\vert ^{p}d\mathbf{x}\right\}  ^{1/p}\qquad\\
\leqslant B\left(  p\right)  \left\{  {\int_{\mathbf{X}}}\left(  {\sum
_{j=1}^{N}}\left\vert \omega_{j}f\left(  x_{j}\right)  -\int_{\mathcal{X}_{j}%
}f\left(  y_{j}\right)  dy_{j}\right\vert ^{2}\right)  ^{p/2}d\mathbf{x}%
\right\}  ^{1/p}.
\end{multline*}

\end{corollary}

\begin{proof}
It suffices to apply the Marcinkiewicz--Zygmund inequality to the independent
random variables
\[
f_{j}\left( \mathbf{x}\right)  =f_{j}\left( x_{1},x_{2},\ldots,x_{N}\right)
=\omega_{j}f\left( x_{j}\right) ,
\]
and observe that
\[
\mathbb{E}\left( f_{j}\right) =\int_{\mathcal{X}_{j}}f\left( x_{j}\right)
dx_{j}%
\]
and
\[
{\displaystyle\sum_{j=1}^{N}} \left( f_{j}-\mathbb{E}\left( f_{j}\right)
\right)  ={\displaystyle\sum_{j=1}^{N}} \left( \omega_{j}f\left( x_{j}\right)
-\int_{\mathcal{X}_{j}}f\left( x_{j}\right) dx_{j}\right)  ={\displaystyle\sum
_{j=1}^{N}} \omega_{j}f\left( x_{j}\right) -\int_{\mathcal{M}}f\left( x\right)
dx.
\]
 \end{proof}

\section{\label{sec GL}Diameter bounded equal measure partition of metric
measure spaces}

In some of the results that follow we shall assume that a metric measure space
$\mathcal{M}$ can be decomposed into a finite disjoint union of sets in the
form $\mathcal{M=X}_{1}\cup\cdots\cup\mathcal{X}_{N}$, with $\omega
_{j}=\left\vert \mathcal{X}_{j}\right\vert \approx N^{-1}$ and $\delta
_{j}=\operatorname*{diam}\left(  \mathcal{X}_{j}\right)  \approx N^{-1/d}$. In
fact, under appropriate assumptions, Gigante and Leopardi in \cite{G-L} proved
the following more precise result.

\begin{theorem}
\label{GigLeo}Let $\mathcal{M}$ be a connected metric measure space with
finite measure with the property that there exist positive constants $d$, $H$
and $K$ such that for every $y\in\mathcal{M}$ and $0<r<\operatorname*{diam}%
(\mathcal{M})$,
\[
Hr^{d} \leqslant\left\vert \left\{  x\in\mathcal{M}:\left\vert x-y\right\vert
< r\right\}  \right\vert \leqslant Kr^{d}.
\]
Then, there exist two constants $c_{1}$ and $c_{2}$, such that for every
sufficiently large $N$ there exists a partition $\mathcal{M}=\mathcal{X}%
_{1}\cup\cdots\cup\mathcal{X}_{N}$ and points $y_{j}\in\mathcal{X}_{j}$ with
$\left|  \mathcal{X}_{j}\right|  = \left|  \mathcal{M} \right|  /N$ and
\[
\left\{  x\in\mathcal{M}:\left\vert x-y_{j}\right\vert < c_{1}N^{-1/d}%
\right\}  \subset\mathcal{X}_{j}\subset\left\{  x\in\mathcal{M}:\left\vert
x-y_{j}\right\vert < c_{2}N^{-1/d} \right\}  .
\]

\end{theorem}

For example, the theorem applies to all compact Riemannian manifolds. 
An algorithmic construction 
in the particular case of the
$2$-dimensional sphere that pays attention to the size of the constant $c_2$
is contained in \cite{Saff} (see also 
\cite{L} for the extension to higher dimensions).

\section{\label{numerical 1} Numerical integration in potential spaces}

In this section we shall study the functional $\mathcal{E}_{\mathbf{x,\omega}%
}$ on the potential space $\mathbb{H}_{p}^{\Phi}\left(  \mathcal{M}\right)  $.
Let
\[
\left\Vert \mathcal{E}_{\mathbf{x,\omega}}\right\Vert _{\Phi,p} =\sup_{f\in\mathbb H^{\Phi}_p(\mathcal M)}\left\{
\frac{\left\vert \mathcal{E}_{\mathbf{x,\omega}}\left(  f\right)  \right\vert
} {\left\Vert f\right\Vert _{\mathbb{H}_{p}^{\Phi}\left(  \mathcal{M}\right)
}} \right\}
\]
be the norm of this functional, also termed \textit{worst case
error}, and the following lemma gives an explicit
formula for it.

\begin{lemma}
\label{Lemma 1} Assume that a measure space $\mathcal{M}$ can be decomposed
into a finite disjoint union of sets in the form $\mathcal{M=X}_{1}\cup
\cdots\cup\mathcal{X}_{N}$, with measure $\left\vert \mathcal{X}%
_{j}\right\vert =\omega_{j}>0$. Assume also that $1\leqslant p\leqslant
+\infty$, $1\leqslant q\leqslant+\infty$, $1/p+1/q=1$, and that for every
$x$,
\[
{\left\{  \int_{\mathcal{M}}\left\vert \Phi\left(  x,y\right)  \right\vert
^{q}dy\right\}  ^{1/q}<+\infty.}%
\]
Finally, assume that
\[
\left\{  \int_{\mathcal{M}}\left(  \int_{\mathcal{M}}\left\vert \Phi\left(
x,y\right)  \right\vert dx\right)  ^{q}dy\right\}  ^{1/q}<+\infty.
\]
Then the functional $\mathcal{E}_{\mathbf{x,\omega}}$ is well defined and
continuous on $H_{p}^{\Phi}\left(  \mathcal{M}\right)  $, and its norm is
\[
\left\Vert \mathcal{E}_{\mathbf{x,\omega}}\right\Vert _{\Phi,p}=\left\{
{\displaystyle\int_{\mathcal{M}}}\left\vert {\displaystyle\sum_{j=1}^{N}%
}{\displaystyle\int_{\mathcal{X}_{j}}}\left(  \Phi(x_{j},y)-\Phi(x,y)\right)
dx\right\vert ^{q}dy\right\}  ^{1/q}.
\]

\end{lemma}

\begin{proof}
For simplicity assume that $q<+\infty$, the case $q=+\infty$ being similar.
Let
\[
f\left( x\right) = {\displaystyle\int_{\mathcal{M}}} \Phi\left( x,y\right)
g\left( y\right) dy
\]
be the potential of a function $g(x)$ in $\mathbb{L}^{p}\left( \mathcal{M}%
\right) $. Since
\[
{\displaystyle\int_{\mathcal{M}}} \left\vert \Phi\left( x,y\right) \right\vert
^{q}dy<+\infty,
\]
$f\left( x\right) $ is pointwise well defined. Since
\[
\int_{\mathcal{M}} \left( \int_{\mathcal{M}}\left\vert \Phi\left( x,y\right)
\right\vert dx\right) ^{q}dy <+\infty,
\]
it follows from Fubini's theorem that $f\left( x\right) $ is integrable, and
\[
{\displaystyle\int_{\mathcal{M}}} f\left( x\right) dx= {\displaystyle\int%
_{\mathcal{M}}} {\displaystyle\int_{\mathcal{M}}} \Phi\left( x,y\right)
g\left( y\right) dy\,dx= {\displaystyle\int_{\mathcal{M}}} g\left( y\right)
{\displaystyle\int_{\mathcal{M}}} \Phi\left( x,y\right) dx\,dy.
\]
This implies that $\mathcal{E}_{\mathbf{x,\omega}}\left( f\right) $ is well
defined. Moreover,
\begin{align*}
\left\vert \mathcal{E}_{\mathbf{x,\omega}}\left( f\right) \right\vert  &
=\left\vert {\displaystyle\sum_{j=1}^{N}} \omega_{j}f\left( x_{j}\right) -
{\displaystyle\int_{\mathcal{M}}} f\left( x\right) dx\right\vert \\
&  =\left\vert {\displaystyle\sum_{j=1}^{N}} \omega_{j} {\displaystyle\int%
_{\mathcal{M}}} \Phi\left( x_{j},y\right) g\left( y\right) dy-
{\displaystyle\int_{\mathcal{M}}} g\left( y\right)  {\displaystyle\int%
_{\mathcal{M}}} \Phi\left( x,y\right) dx\,dy\right\vert \\
&  =\left\vert {\displaystyle\int_{\mathcal{M}}} \left(  {\displaystyle\sum
_{j=1}^{N}} \int_{\mathcal{X}_{j}}\left( \Phi\left( x_{j},y\right)
-\Phi\left(  x,y\right) \right) dx\right) g\left( y\right) dy\right\vert \\
&  \leqslant\left\{  {\displaystyle\int_{\mathcal{M}}} \left\vert
g(y)\right\vert ^{p}dy\right\} ^{1/p}\left\{  {\displaystyle\int_{\mathcal{M}%
}} \left\vert {\displaystyle\sum_{j=1}^{N}} \int_{\mathcal{X}_{j}}\left(
\Phi\left( x_{j},y\right)  -\Phi\left( x,y\right) \right) dx\right\vert
^{q}dy\right\} ^{1/q}.
\end{align*}
Taking the infimum as $g\left( y\right) $ varies among all possible functions
in $\mathbb{L}^{p}\left( \mathcal{M}\right) $ with potential $f\left( x\right)
$, one obtains
\[
\left\Vert \mathcal{E}_{\mathbf{x,\omega}}\right\Vert _{\Phi,p}\leqslant
\left\{  {\displaystyle\int_{\mathcal{M}}} \left\vert {\displaystyle\sum
_{j=1}^{N}} \int_{\mathcal{X}_{j}}\left( \Phi\left( x_{j},y\right)
-\Phi\left( x,y\right) \right) dx\right\vert ^{q}dy\right\} ^{1/q}.
\]
Conversely, using the standard argument for $L^p-L^q$ duality, set
\[
F\left( y\right) = {\displaystyle\sum_{j=1}^{N}} \int_{\mathcal{X}_{j}}\left(
\Phi\left( x_{j},y\right)  -\Phi\left( x,y\right) \right) dx,
\]
\[
g\left( y\right) =\left\{
\begin{array}
[c]{ll}%
\overline{F\left( y\right) }\left\vert F\left( y\right) \right\vert ^{q/p-1} &
\text{if $F\left( y\right) \neq0$},\\
0 & \text{if $F\left( y\right) =0$},
\end{array}
\right.
\]
and
\[
f\left( x\right) = {\displaystyle\int_{\mathcal{M}}} \Phi\left( x,y\right)
g\left( y\right) dy.
\]
Then
\begin{align*}
\left\vert \mathcal{E}_{\mathbf{x,\omega}}\left( f\right) \right\vert  &
=\left\vert {\displaystyle\sum_{j=1}^{N}} \omega_{j}f\left( x_{j}\right) -
{\displaystyle\int_{\mathcal{M}}} f(x)dx\right\vert =\left\vert
{\displaystyle\int_{\mathcal{M}}} F\left( y\right) g\left( y\right)
dy\right\vert \\
&  ={\displaystyle\int_{\mathcal{M}}} \left\vert F\left( y\right) \right\vert
^{1+q/p}dy =\left\{  {\displaystyle\int_{\mathcal{M}}} \left\vert F\left(
y\right) \right\vert ^{q}dy\right\} ^{1/q} \left\{  {\displaystyle\int%
_{\mathcal{M}}} \left\vert F\left( y\right) \right\vert ^{q}dy\right\}
^{1/p}\\
&  =\left\{  {\displaystyle\int_{\mathcal{M}}} \left\vert F\left( y\right)
\right\vert ^{q}dy\right\} ^{1/q} \left\{  {\displaystyle\int_{\mathcal{M}}}
\left\vert g\left( y\right) \right\vert ^{p}dy\right\} ^{1/p} \geqslant
\left\{  {\displaystyle\int_{\mathcal{M}}} \left\vert F\left( y\right)
\right\vert ^{q}dy\right\} ^{1/q} \left\Vert f\right\Vert _{\mathbb{H}%
_{p}^{\Phi}}.
\end{align*}
This implies that
\[
\left\Vert \mathcal{E}_{\mathbf{x,\omega}}\right\Vert _{\Phi,p} \geqslant
\left\{  {\displaystyle\int_{\mathcal{M}}} \left\vert F\left( y\right)
\right\vert ^{q}dy\right\} ^{1/q}.
\]
 \end{proof}

\begin{theorem}
\label{Theorem 3} Assume that a measure space $\mathcal{M}$ can be decomposed
into a finite disjoint union of sets in the form $\mathcal{M=X}_{1}\cup
\cdots\cup\mathcal{X}_{N}$, with measure $\left\vert \mathcal{X}%
_{j}\right\vert =\omega_{j}>0$. Assume also that $1\leqslant p\leqslant
+\infty$, $1\leqslant q\leqslant+\infty$, $1/p+1/q=1$, and that for every
$x$,
\[
\left\{  \int_{\mathcal{M}}\left\vert \Phi\left(  x,y\right)  \right\vert
^{q}dy\right\}  ^{1/q}<+\infty.
\]
Finally, assume that
\[
\left\{  \int_{\mathcal{M}}\left(  \int_{\mathcal{M}}\left\vert \Phi\left(
x,y\right)  \right\vert dx\right)  ^{q}dy\right\}  ^{1/q}<+\infty.
\]
Define
\[
\Gamma\left(  \Phi\right)  ={\sum_{j=1}^{N}}\left\{  {\int_{\mathcal{X}_{j}%
}\int_{\mathcal{M}}}\left\vert {\int_{\mathcal{X}_{j}}}\left(  \Phi
(x_{j},y)-\Phi(z_{j},y)\right)  dz_{j}\right\vert ^{q}dy\frac{dx_{j}}%
{\omega_{j}}\right\}  ^{1/q},
\]
and
\[
\Delta\left(  \Phi\right)  =\left\{  {\int_{\mathbf{X}}\int_{\mathcal{M}}%
}\left(  {\sum_{j=1}^{N}}\left\vert {\int_{\mathcal{X}_{j}}}\left(  \Phi
(x_{j},y)-\Phi(z_{j},y)\right)  dz_{j}\right\vert ^{2}\right)  ^{q/2}%
dy\,d\mathbf{x}\right\}  ^{1/q}.
\]
Then for every $1\leqslant p\leqslant+\infty$,
\begin{equation}
\left\{  {\int_{\mathbf{X}}}\left\Vert \mathcal{E}_{\mathbf{x},\mathbf{\omega
}}\right\Vert _{\Phi,p}^{q}d\mathbf{x}\right\}  ^{1/q}\leqslant\Gamma\left(
\Phi\right)  ,\label{Est1}%
\end{equation}
and for every $1<p\leqslant+\infty$,
\begin{equation}
A\left(  q\right)  \Delta\left(  \Phi\right)  \leqslant\left\{  {\int%
_{\mathbf{X}}}\left\Vert \mathcal{E}_{\mathbf{x},\mathbf{\omega}}\right\Vert
_{\Phi,p}^{q}d\mathbf{x}\right\}  ^{1/q}\leqslant B\left(  q\right)
\Delta\left(  \Phi\right)  .\label{Est2}%
\end{equation}
In particular, there exist choices of nodes $\left\{  x_{j}\right\}  $ with
the property that for every function $f\left(  x\right)  $ in the potential
space $\mathbb{H}_{p}^{\Phi}\left(  \mathcal{M}\right)  $,
\[
\left\vert {\sum_{j=1}^{N}}\omega_{j}f\left(  x_{j}\right)  -{\int%
_{\mathcal{M}}}f(x)dx\right\vert \leqslant\left\{
\begin{array}
[c]{ll}%
\Gamma\left(  \Phi\right)  \left\Vert f\right\Vert _{\mathbb{H}_{p}^{\Phi
}\left(  \mathcal{M}\right)  } & \text{for every $1\leqslant p\leqslant
+\infty$},\\
B\left(  q\right)  \Delta\left(  \Phi\right)  \left\Vert f\right\Vert
_{\mathbb{H}_{p}^{\Phi}\left(  \mathcal{M}\right)  } & \text{for every
$1<p\leqslant+\infty$}.
\end{array}
\right.
\]

\end{theorem}

The constants $A\left(  q\right)  $ and $B\left(  q\right)  $ are the best
constants in the Marcinkiewicz--Zygmund inequality. The constants
$\Gamma\left(  \Phi\right)  $ and $\Delta\left(  \Phi\right)  $ are related to
the smoothness of the kernel $\Phi\left(  x,y\right)  $. These last constants
could be estimated in terms of Sobolev norms. However, in the applications,
the estimates in terms of Sobolev norms are not always optimal, and it is more
convenient to keep the above complicated expressions. Finally, since $B\left(
q\right)  \rightarrow+\infty$ as $p\rightarrow1+$, the interest of the
estimate (\ref{Est1}) is when $p\rightarrow1+$.

\begin{proof}
By Lemma \ref{Lemma 1} and the triangle inequality,
\begin{align*}
\left\{  {\int_{\mathbf{X}}}\left\Vert \mathcal{E}_{\mathbf{x},\mathbf{\omega
}}\right\Vert _{\Phi,p}^{q}d\mathbf{x}\right\}  ^{1/q} &  =\left\{
{\int_{\mathbf{X}}\int_{\mathcal{M}}}\left\vert {\sum_{j=1}^{N}}%
{\int_{\mathcal{X}_{j}}}\left(  \Phi(x_{j},y)-\Phi(z_{j},y)\right)
dz_{j}\right\vert ^{q}\,dy\,d\mathbf{x}\right\}  ^{1/q}\\
&  \leqslant{\sum_{j=1}^{N}}\left\{  {\int_{\mathcal{X}_{j}}\int_{\mathcal{M}%
}}\left\vert {\int_{\mathcal{X}_{j}}}\left(  \Phi(x_{j},y)-\Phi(z_{j}%
,y)\right)  dz_{j}\right\vert ^{q}dy\,\dfrac{dx_{j}}{\omega_{j}}\right\}
^{1/q},
\end{align*}
where we have used the fact that for a function depending only on $x_j$,
integration on $\mathbf X$ coincides with integration on $\mathcal X_j$.
This gives the proof with $\Gamma\left(  \Phi\right)  $. The proof with
$\Delta\left(  \Phi\right)  $ is similar, with the crucial difference that we replace the triangle inequality
with the Marcinkiewicz-Zygmund inequality. Indeed, by Corollary \ref{Cor-MZ},
\begin{multline*}
\left\{  {\int_{\mathbf{X}}}\left\Vert \mathcal{E}_{\mathbf{x},\mathbf{\omega
}}\right\Vert _{\Phi,p}^{q}d\mathbf{x}\right\}  ^{1/q}   =\left\{
{\int_{\mathcal{M}}\int_{\mathbf{X}}}\left\vert {\sum_{j=1}^{N}\int%
_{\mathcal{X}_{j}}}\left(  \Phi(x_{j},y)-\Phi(z_{j},y)\right)  dz_{j}%
\right\vert ^{q}d\mathbf{x}\,dy\right\}  ^{1/q}\\
\leqslant B\left(  q\right)  \left\{  {\int_{\mathcal{M}}\int_{\mathbf{X}}%
}\left(  {\sum_{j=1}^{N}}\left\vert {\int_{\mathcal{X}_{j}}}\left(  \Phi
(x_{j},y)-\Phi(z_{j},y)\right)  dz_{j}\right\vert ^{2}\right)  ^{q/2}%
d\mathbf{x}\,dy\right\}  ^{1/q}.
\end{multline*}
The proof for the lower bound is similar.
 \end{proof}




The following corollary is a slightly generalized version of Theorem I in the Introduction.

\begin{corollary}
\label{Corollary 3} Let $\mathcal{M}$ be a metric measure space with the
property that there exist $d$ and $c$ such that for every $y\in\mathcal{M}$
and $r>0$,
\[
\left\vert \left\{  x\in\mathcal{M}:\left\vert x-y\right\vert \leqslant
r\right\}  \right\vert \leqslant cr^{d}.
\]
Assume also that $\mathcal{M}$ can be decomposed into a finite disjoint union
of sets in the form $\mathcal{M=X}_{1}\cup\cdots\cup\mathcal{X}_{N}$, with
$\omega_{j}=\left\vert \mathcal{X}_{j}\right\vert \approx N^{-1}$ and
$\delta_{j}=\operatorname*{diam}\left(  \mathcal{X}_{j}\right)  \approx
N^{-1/d}$. Assume that for some $\varepsilon>0$ and $0<\alpha<d$,
\[
\left\vert \Phi\left(  x,y\right)  \right\vert \leqslant c\left\vert
x-y\right\vert ^{\alpha-d}%
\]
for every $x$ and $y$, and
\[
\left\vert \Phi\left(  x,y\right)  -\Phi\left(  z,y\right)  \right\vert
\leqslant c\left\vert x-z\right\vert ^{\varepsilon} \left\vert x-y\right\vert
^{\alpha-d-\varepsilon}%
\]
if $\left\vert x-y\right\vert \geqslant2\left\vert x-z\right\vert $. Finally,
assume that $1<p\leqslant+\infty$, $1/p+1/q=1$ and $d/p<\alpha<d$. Then
\[
\left\{  {\displaystyle\int_{\mathbf{X}}} \left\Vert \mathcal{E}%
_{\mathbf{x},\mathbf{\omega}}\right\Vert_{\Phi,p} ^{q} d\mathbf{x}\right\}  ^{1/q}
\leqslant\left\{
\begin{array}
[c]{ll}%
cN^{-\alpha/d} & \text{if $\alpha<d/2+\varepsilon$},\\
cN^{-1/2-\varepsilon/d}\left(  \log N\right)  ^{1/2} & \text{if $\alpha
=d/2+\varepsilon$},\\
cN^{-1/2-\varepsilon/d} & \text{if $\alpha>d/2+\varepsilon$}.
\end{array}
\right.
\]

\end{corollary}

\begin{proof}
The assumption $\alpha>d/p$ ensures that the kernel $\Phi\left( x,y\right) $
is $q$ integrable and satisfies the hypotheses of Theorem \ref{Theorem 3}. It
then suffices to estimate
\[
\Delta\left( \Phi\right) =\left\{  {\displaystyle\int_{\mathcal{M}}}
{\displaystyle\int_{\mathbf{X}}} \left(  {\displaystyle\sum_{j=1}^{N}}
\left\vert {\displaystyle\int_{\mathcal{X}_{j}}} \left( \Phi(x_{j}%
,y)-\Phi(z_{j},y)\right) dz_{j}\right\vert ^{2}\right) ^{q/2} d\mathbf{x}\,dy%
\right\} ^{1/q}.
\]
If $\operatorname*{dist}\left( y,\mathcal{X}_{j}\right) \leqslant2\delta_{j}$,
then for every $x_{j}$ in $\mathcal{X}_{j}$,
\begin{align*}
\left\vert {\displaystyle\int_{\mathcal{X}_{j}}} \left( \Phi(x_{j}%
,y)-\Phi(z_{j},y)\right) dz_{j}\right\vert  &  \leqslant c {\displaystyle\int%
_{\mathcal{X}_{j}}} \left( \left\vert x_{j}-y\right\vert ^{\alpha-d}
+\left\vert z_{j}-y\right\vert ^{\alpha-d}\right) dz_{j}\\
&  \leqslant c\omega_{j}\left\vert x_{j}-y\right\vert ^{\alpha-d} +c\delta
_{j}^{\alpha} \leqslant cN^{-1}\left\vert x_{j}-y\right\vert ^{\alpha-d}.
\end{align*}
If $\operatorname*{dist}\left( y,\mathcal{X}_{j}\right) \geqslant2\delta_{j}$,
then
\begin{align*}
\left\vert {\displaystyle\int_{\mathcal{X}_{j}}} \left( \Phi(x_{j}%
,y)-\Phi(z_{j},y)\right) dz_{j}\right\vert  &  \leqslant c {\displaystyle\int%
_{\mathcal{X}_{j}}} \left\vert x_{j}-z_{j}\right\vert ^{\varepsilon}
\left\vert x_{j}-y\right\vert ^{\alpha-d-\varepsilon}dz_{j}\\
&  \leqslant c\delta_{j}^{\varepsilon}\omega_{j} \left\vert x_{j}-y\right\vert
^{\alpha-d-\varepsilon} \leqslant cN^{-1-\varepsilon/d} \left\vert
x_{j}-y\right\vert ^{\alpha-d-\varepsilon}.
\end{align*}
Hence
\begin{align*}
&  \left\{  {\displaystyle\int_{\mathcal{M}}} {\displaystyle\int_{\mathbf{X}}}
\left( \sum_{j=1}^{N}\left\vert \int_{\mathcal{X}_{j}} \left( \Phi
(x_{j},y)-\Phi(z_{j},y)\right) dz_{j}\right\vert ^{2}\right) ^{q/2}
d\mathbf{x}\,dy\right\} ^{1/q}\\
& \quad\leqslant c\left\{  {\displaystyle\int_{\mathcal{M}}}
{\displaystyle\int_{\mathbf{X}}} \left( N^{-2} {\displaystyle\sum
_{j:\operatorname*{dist}\left( y,\mathcal{X}_{j}\right)  \leqslant2\delta_{j}%
}} \left\vert x_{j}-y\right\vert ^{2\alpha-2d}\right) ^{q/2} d\mathbf{x}%
\,dy\right\} ^{1/q}\\
& \quad\qquad+c\left\{  {\displaystyle\int_{\mathcal{M}}} {\displaystyle\int%
_{\mathbf{X}}} \left( N^{-2-2\varepsilon/d} {\displaystyle\sum
_{j:\operatorname*{dist}\left( y,\mathcal{X}_{j}\right) >2\delta_{j}}}
\left\vert x_{j}-y\right\vert ^{2\alpha-2d-2\varepsilon}\right) ^{q/2}
d\mathbf{x}\,dy\right\} ^{1/q}.
\end{align*}
Under the assumption that $\operatorname*{diam}\left( \mathcal{X}_{j}\right)
\approx N^{-1/d}$, there is only a bounded number of $\mathcal{X}_{j}$ with
$\operatorname*{dist}\left( y,\mathcal{X}_{j}\right)  \leqslant
2\operatorname*{diam}\left( \mathcal{X}_{j}\right) $. Hence
\begin{align*}
&  \left\{  {\displaystyle\int_{\mathcal{M}}} {\displaystyle\int_{\mathbf{X}}}
\left( N^{-2} {\displaystyle\sum_{j:\operatorname*{dist}\left( y,\mathcal{X}%
_{j}\right)  \leqslant2\delta_{j}}} \left\vert x_{j}-y\right\vert
^{2\alpha-2d}\right) ^{q/2} d\mathbf{x}\,dy\right\} ^{1/q}\\
& \quad\leqslant c\left\{  {\displaystyle\int_{\mathcal{M}}}
{\displaystyle\int_{\mathbf{X}}} {\displaystyle\sum_{j=1}^{N}} \left(
N^{-2}\chi_{\left\{ \operatorname*{dist}\left( y,\mathcal{X}_{j}\right)
\leqslant2\delta_{j}\right\} }\left( y\right)  \left\vert x_{j}-y\right\vert
^{2\alpha-2d}\right) ^{q/2}d\mathbf{x}\,dy\right\} ^{1/q}\\
& \quad\leqslant c\left\{ N^{-q} {\displaystyle\sum_{j=1}^{N}}
{\displaystyle\int_{\mathcal{X}_{j}}} {\displaystyle\int_{\left\{ \left\vert
y-x_{j}\right\vert \leqslant cN^{-1/d}\right\} }} \left\vert x_{j}%
-y\right\vert ^{\alpha q-dq} dy\frac{dx_{j}}{\omega_{j}}\right\} ^{1/q}
\leqslant cN^{-\alpha/d}.
\end{align*}
Moreover,
\begin{align*}
&  N^{-2-2\varepsilon/d} {\displaystyle\sum_{j:\operatorname*{dist}\left(
y,\mathcal{X}_{j}\right) >2\delta_{j}}} \left\vert x_{j}-y\right\vert
^{2\alpha-2d-2\varepsilon}\\
& \quad\leqslant cN^{-1-2\varepsilon/d} {\displaystyle\sum
_{j:\operatorname*{dist}\left( y,\mathcal{X}_{j}\right) >2\delta_{j}}}
\omega_{j}\left\vert x_{j}-y\right\vert ^{2\alpha-2d-2\varepsilon}\\
& \quad\leqslant cN^{-1-2\varepsilon/d} {\displaystyle\sum
_{j:\operatorname*{dist}\left( y,\mathcal{X}_{j}\right) >2\delta_{j}}}
\int_{\mathcal{X}_{j}}\left\vert x-y\right\vert ^{2\alpha-2d-2\varepsilon}dx\\
& \quad\leqslant cN^{-1-2\varepsilon/d} {\displaystyle\int_{\left\{ \left\vert
x-y\right\vert >cN^{-1/d}\right\} }} \left\vert x-y\right\vert ^{2\alpha
-2d-2\varepsilon}dx\\
& \quad\leqslant\left\{
\begin{array}
[c]{ll}%
cN^{-2\alpha/d} & \text{if $\alpha<\varepsilon+d/2$},\\
cN^{-1-2\varepsilon/d}\log N & \text{if $\alpha=\varepsilon+d/2$},\\
cN^{-1-2\varepsilon/d} & \text{if $\alpha>\varepsilon+d/2$}.
\end{array}
\right.
\end{align*}
Hence
\begin{align*}
&  \left\{  {\displaystyle\int_{\mathcal{M}}} {\displaystyle\int_{\mathbf{X}}}
\left( N^{-2-2\varepsilon/d} {\displaystyle\sum_{j:\operatorname*{dist}\left(
y,\mathcal{X}_{j}\right) >2\delta_{j}}} \left\vert x_{j}-y\right\vert
^{2\alpha-2d-2\varepsilon}\right) ^{q/2} d\mathbf{x}\,dy\right\} ^{1/q}\\
& \quad\leqslant\left\{
\begin{array}
[c]{ll}%
cN^{-\alpha/d} & \text{if $\alpha<\varepsilon+d/2$},\\
cN^{-1/2-\varepsilon/d}\left( \log N\right) ^{1/2} & \text{if $\alpha
=\varepsilon+d/2$},\\
cN^{-1/2-\varepsilon/d} & \text{if $\alpha>\varepsilon+d/2$}.
\end{array}
\right.
\end{align*}
 \end{proof}

The following result shows that under some natural assumptions on the kernel,
the mean value estimate in the above corollary is essentially sharp, an is a slightly 
generalized version of Theorem II in the Introduction.

\begin{corollary}
\label{Corollary 4} Let $\mathcal{M}$ be a metric measure space with the
property that there exist $H,K$ and $d$ such that for every $y\in\mathcal{M} $
and $0<r<r_{0}$,
\[
Hr^{d}\leqslant\left\vert \left\{  x\in\mathcal{M}:\left\vert x-y\right\vert
\leqslant r\right\}  \right\vert \leqslant Kr^{d}.
\]
Assume also that $\mathcal{M}$ can be decomposed into a finite disjoint union
of sets in the form $\mathcal{M=X}_{1}\cup\cdots\cup\mathcal{X}_{N}$, with
$\omega_{j}=\left\vert \mathcal{X}_{j}\right\vert \approx N^{-1}$ and
$\delta_{j}=\operatorname*{diam}\left(  \mathcal{X}_{j}\right)  \approx
N^{-1/d} $. Suppose that there exists $0<\alpha<d$ and $\varepsilon>0$, such
that for any $j=1,\ldots,N$ and any $z\in\mathcal{X}_{j}$, and for any $y$
such that $\operatorname{dist}\left(  y,\mathcal{X}_{j}\right)  \geqslant
2\delta_{j}$,
\[
\int_{\mathcal{X}_{j}}\left\vert \Phi\left(  x,y\right)  -\Phi\left(
z,y\right)  \right\vert dx\geqslant cN^{-1-\varepsilon/d}\left(
\operatorname{dist}\left(  y,\mathcal{X}_{j}\right)  \right)  ^{\alpha
-d-\varepsilon}.
\]
Suppose also that for any $y\in M$, the function $x\mapsto\Phi\left(
x,y\right)  $ is continuous in $x\neq y$. Finally, assume that $1<p\leqslant
+\infty$, $1/p+1/q=1$ and $d/p<\alpha<d$. Then
\[
\left\{  {\int_{\mathbf{X}}}\left\Vert \mathcal{E}_{\mathbf{x},\mathbf{\omega
}}\right\Vert_{\Phi,p} ^{q}d\mathbf{x}\right\}  ^{1/q}\geqslant\left\{
\begin{array}
[c]{ll}%
cN^{-\alpha/d} & \text{if $\alpha<d/2+\varepsilon$},\\
cN^{-1/2-\varepsilon/d}\left(  \log N\right)  ^{1/2} & \text{if $\alpha
=d/2+\varepsilon$},\\
cN^{-1/2-\varepsilon/d} & \text{if $\alpha>d/2+\varepsilon$}.
\end{array}
\right.
\]

\end{corollary}

\begin{proof}
It follows from Lemma \ref{Lemma 1} that
\begin{align*}
\left\{  {\int_{\mathbf{X}}}\left\Vert \mathcal{E}_{\mathbf{x},\mathbf{\omega
}}\right\Vert ^{q}d\mathbf{x}\right\}  ^{1/q} &  =\left\{  {\int_{\mathbf{X}}%
}{\int_{\mathcal{M}}}\left\vert {\sum_{j=1}^{N}}{\int_{\mathcal{X}_{j}}%
}\left(  \Phi(x_{j},y)-\Phi(z_{j},y)\right)  dz_{j}\right\vert ^{q}%
dy\,d\mathbf{x}\right\}  ^{1/q}\\
&  \geqslant\left\vert \mathcal{M}\right\vert ^{-1/p}{\int_{\mathbf{X}}}%
{\int_{\mathcal{M}}}\left\vert {\sum_{j=1}^{N}}{\int_{\mathcal{X}_{j}}}\left(
\Phi(x_{j},y)-\Phi(z_{j},y)\right)  dz_{j}\right\vert dy\,d\mathbf{x}.
\end{align*}
By Corollary \ref{Cor-MZ}, this is bounded from below by
\begin{align*}
&  \left\vert \mathcal{M}\right\vert ^{-1/p}A\left(  1\right)  {\int%
_{\mathcal{M}}}{\int_{\mathbf{X}}}\left(  {\sum_{j}}\left\vert {\int%
_{\mathcal{X}_{j}}}\left(  \Phi(x_{j},y)-\Phi(z_{j},y)\right)  dz_{j}%
\right\vert ^{2}\right)  ^{1/2}d\mathbf{x}\,dy\\
&  \quad\geqslant\left\vert \mathcal{M}\right\vert ^{-1/p}A\left(  1\right)
{\int_{\mathcal{M}}}{\int_{\mathbf{X}}}\left(  {\sum_{j:\operatorname{dist}%
\left(  y,\mathcal{X}_{j}\right)  \geqslant2\delta_{j}}}\left\vert
{\int_{\mathcal{X}_{j}}}\left(  \Phi(x_{j},y)-\Phi(z_{j},y)\right)
dz_{j}\right\vert ^{2}\right)  ^{1/2}d\mathbf{x}\,dy.
\end{align*}
By the continuity of $z_{j}\rightarrow\Phi(z_{j},y)$, there exists a point
$x_{j}^{\ast}$, depending on $y$, such that
\[
{\int_{\mathcal{X}_{j}}}\Phi(z_{j},y)dz_{j}=\omega_{j}\Phi(x_{j}^{\ast},y).
\]
Thus
\begin{align*}
&  {\int_{\mathbf{X}}}\left(  {\sum_{j:\operatorname{dist}\left(
y,\mathcal{X}_{j}\right)  \geqslant2\delta_{j}}}\left\vert {\int%
_{\mathcal{X}_{j}}}\left(  \Phi(x_{j},y)-\Phi(z_{j},y)\right)  dz_{j}%
\right\vert ^{2}\right)  ^{1/2}d\mathbf{x}\\
&  \quad={\int_{\mathbf{X}}}\left(  {\sum_{j:\operatorname{dist}\left(
y,\mathcal{X}_{j}\right)  \geqslant2\delta_{j}}}\omega_{j}^{2}\left\vert
\Phi(x_{j},y)-\Phi(x_{j}^{\ast},y)\right\vert ^{2}\right)  ^{1/2}d\mathbf{x}.
\end{align*}
For any two positive sequences $\left\{  \alpha_{j}\right\}  $ and $\left\{
\beta_{j}\right\}  $ we clearly have
\[
\sum_{j}\alpha_{j}^{2}\beta_{j}\leqslant\left(  \sum_{j}\alpha_{j}^{2}\right)
^{1/2}\left(  \sum_{j}\alpha_{j}^{2}\beta_{j}^{2}\right)  ^{1/2}.
\]
Thus if $v_{j}$ is a point of the closure of $\mathcal{X}_{j}$ that minimizes
the distance from $y$,
\begin{align*}
&  {\int_{\mathbf{X}}}\left(  {\sum_{j:\operatorname{dist}\left(
y,\mathcal{X}_{j}\right)  \geqslant2\delta_{j}}}\omega_{j}^{2}\left\vert
v_{j}-y\right\vert ^{2\alpha-2d-2\varepsilon}\frac{\left\vert \Phi
(x_{j},y)-\Phi(x_{j}^{\ast},y)\right\vert ^{2}}{\left\vert v_{j}-y\right\vert
^{2\alpha-2d-2\varepsilon}}\right)  ^{1/2}d\mathbf{x}\\
&  \quad\geqslant{\int_{\mathbf{X}}}\left(  {\sum_{j:\operatorname{dist}%
\left(  y,\mathcal{X}_{j}\right)  \geqslant2\delta_{j}}}\omega_{j}%
^{2}\left\vert v_{j}-y\right\vert ^{2\alpha-2d-2\varepsilon}\right)
^{-1/2}\times\\
&  \quad\qquad\qquad\times{\sum_{j:\operatorname{dist}\left(  y,\mathcal{X}%
_{j}\right)  \geqslant2\delta_{j}}}\omega_{j}^{2}\left\vert v_{j}-y\right\vert
^{2\alpha-2d-2\varepsilon}\frac{\left\vert \Phi(x_{j},y)-\Phi(x_{j}^{\ast
},y)\right\vert }{\left\vert v_{j}-y\right\vert ^{\alpha-d-\varepsilon}%
}\,d\mathbf{x}\\
&  \quad=\left(  {\sum_{j:\operatorname{dist}\left(  y,\mathcal{X}_{j}\right)
\geqslant2\delta_{j}}}\omega_{j}^{2}\left\vert v_{j}-y\right\vert
^{2\alpha-2d-2\varepsilon}\right)  ^{-1/2}\times\\
&  \quad\qquad\qquad\times{\sum_{j:\operatorname{dist}\left(  y,\mathcal{X}%
_{j}\right)  \geqslant2\delta_{j}}}\omega_{j}^{2}\left\vert v_{j}-y\right\vert
^{\alpha-d-\varepsilon}\left(  {\int_{\mathbf{X}}}\left\vert \Phi
(x_{j},y)-\Phi(x_{j}^{\ast},y)\right\vert d\mathbf{x}\right) \\
&  \quad\geqslant c\left(  {\sum_{j:\operatorname{dist}\left(  y,\mathcal{X}%
_{j}\right)  \geqslant2\delta_{j}}}\omega_{j}^{2}\left\vert v_{j}-y\right\vert
^{2\alpha-2d-2\varepsilon}\right)  ^{-1/2}{\sum_{j:\operatorname{dist}\left(
y,\mathcal{X}_{j}\right)  \geqslant2\delta_{j}}}\omega_{j}^{2}\left\vert
v_{j}-y\right\vert ^{2\alpha-2d-2\varepsilon}N^{-\varepsilon/d}\\
&  \quad=c\left(  {\sum_{j:\operatorname{dist}\left(  y,\mathcal{X}%
_{j}\right)  \geqslant2\delta_{j}}}\omega_{j}^{2}\left\vert v_{j}-y\right\vert
^{2\alpha-2d-2\varepsilon}\right)  ^{1/2}N^{-\varepsilon/d}.
\end{align*}
The desired result now follows from the estimates
\[
{\sum_{j:\operatorname{dist}\left(  y,\mathcal{X}_{j}\right)  \geqslant
2\delta_{j}}}\omega_{j}^{2}\left\vert v_{j}-y\right\vert ^{2\alpha
-2d-2\varepsilon}\geqslant c\left\{
\begin{array}
[c]{ll}%
N^{-2\alpha/d+2\varepsilon/d} & \text{if $\alpha<d/2+\varepsilon$},\\
N^{-1}\log N & \text{if $\alpha=d/2+\varepsilon$},\\
N^{-1} & \text{if $\alpha>d/2+\varepsilon$}.
\end{array}
\right.
\]
Indeed,
\[
{\sum_{j:\operatorname{dist}\left(  y,\mathcal{X}_{j}\right)  \geqslant
2\delta_{j}}}\omega_{j}^{2}\left\vert v_{j}-y\right\vert ^{2\alpha
-2d-2\varepsilon}\geqslant cN^{-1}{\int_{\left\{  \left\vert x-y\right\vert
>cN^{-1/d}\right\}  }}\left\vert x-y\right\vert ^{2\alpha-2d-2\varepsilon}dx.
\]
If $\lambda\geqslant\left(  2K/H\right)  ^{1/d}$, then for every centre $y$
and $r<r_{0}/\lambda$,
\[
\left\vert \left\{  x\in\mathcal{M}:r<\left\vert x-y\right\vert \leqslant
\lambda r\right\}  \right\vert \geqslant Kr^{d}.
\]
This gives
\begin{align*}
&  N^{-1}{\int_{\left\{  \left\vert x-y\right\vert >cN^{-1/d}\right\}  }%
}\left\vert x-y\right\vert ^{2\alpha-2d-2\varepsilon}dx\\
&  \quad\geqslant N^{-1}\sum_{k=0}^{\left[  \log_{\lambda}\left(  r_{0}%
N^{1/d}\right)  \right]  -1}{\int_{\left\{  \lambda^{k}N^{-1/d}\leqslant
\left\vert x-y\right\vert <\lambda^{k+1}N^{-1/d}\right\}  }}\left\vert
x-y\right\vert ^{2\alpha-2d-2\varepsilon}dx\\
&  \quad\geqslant N^{-1}\sum_{k=0}^{\left[  \log_{\lambda}\left(  r_{0}%
N^{1/d}\right)  \right]  -1}\left(  \lambda^{k}N^{-1/d}\right)  ^{2\alpha
-2d-2\varepsilon}K\left(  \lambda^{k}N^{-1/d}\right)  ^{d}\\
&  \quad=KN^{-2\alpha/d+2\varepsilon/d}\sum_{k=0}^{\left[  \log_{\lambda
}\left(  r_{0}N^{1/d}\right)  \right]  -1}\lambda^{k\left(  2\alpha
-d-2\varepsilon\right)  }\\
&  \quad\geqslant c\left\{
\begin{array}
[c]{ll}%
N^{-2\alpha/d+2\varepsilon/d} & \text{if $\alpha<d/2+\varepsilon$},\\
N^{-1}\log N & \text{if $\alpha=d/2+\varepsilon$},\\
N^{-1} & \text{if $\alpha>d/2+\varepsilon$}.
\end{array}
\right.
\end{align*}
 \end{proof}

\begin{example}\label{example}
Let $\mathcal{M}$ be a $d$-dimensional compact Riemannian manifold. Let
$\left\{  \lambda^{2}\right\}  $ and $\left\{  \varphi_{\lambda}(x)\right\}  $
be the eigenvalues and a complete orthonormal system of eigenfunctions of the
Laplace Beltrami operator $\Delta$ respectively. Every tempered distribution
on $\mathcal{M}$ has Fourier transform and series
\[
\mathcal{F}f(\lambda)={\int_{\mathcal{M}}}f(y)\overline{\varphi_{\lambda}%
(y)}dy\quad\mbox{and}\quad f(x)={\sum\limits_{\lambda}}\mathcal{F}%
f(\lambda)\varphi_{\lambda}(x).
\]
The Bessel kernel $B^{\alpha}(x,y)$, $-\infty<\alpha<+\infty$, is a
distribution defined by the expansion
\[
B^{\alpha}(x,y)={\sum_{\lambda}}\left(  1+\lambda^{2}\right)  ^{-\alpha
/2}\varphi_{\lambda}(x)\overline{\varphi_{\lambda}(y)}.
\]
A distribution $f(x)$ is the Bessel potential of a distribution $g(x)$ if
\[
f(x)={\int_{\mathcal{M}}}B^{\alpha}(x,y)g(y)dy={\sum_{\lambda}}\left(
1+\lambda^{2}\right)  ^{-\alpha/2}\mathcal{F}g(\lambda)\varphi_{\lambda}(x).
\]
Bessel potentials of functions in $\mathbb{L}^{p}\left(  \mathcal{M}\right)  $
define the fractional Sobolev space $\mathbb{H}_{p}^{\alpha}\left(
\mathcal{M}\right)  $. If $0<\alpha<d$, then the Bessel kernel satisfies the
estimates
\[
\left\vert B^{\alpha}(x,y)\right\vert \leqslant c\left\vert x-y\right\vert
^{\alpha-d}\quad\mbox{and}\quad\left\vert \nabla B^{\alpha}(x,y)\right\vert
\leqslant c\left\vert x-y\right\vert ^{\alpha-d-1}.
\]
See \cite[Lemmas 2.5 and 2.6]{BCCGST}. In particular, Corollary \ref{Corollary 3} applies with
 $\Phi(x,y)=B^{\alpha}(x,y)$ and $\varepsilon=1$. Indeed, using the Hadamard parametrix for the wave equation,
see, for example, \cite{BCo}, one can prove a more precise result: there is a
smooth positive function $C\left(  y\right)  $ and positive constants $\eta$
and $c$ such that
\[
B^{\alpha}\left(  x,y\right)  =C\left(  y\right)  \left\vert x-y\right\vert
^{\alpha-d}+E\left(  x,y\right)  ,
\]
with
\[
\left\vert E\left(  x,y\right)  \right\vert \leqslant c\left\vert
x-y\right\vert ^{\alpha-d+\eta}\quad\mbox{and}\quad\left\vert \nabla E\left(
x,y\right)  \right\vert \leqslant c\left\vert x-y\right\vert ^{\alpha
-d-1+\eta}.
\]
It then follows that Corollary \ref{Corollary 4} applies also. See
\cite[Theorems 24 and 25]{BSSW} for the case of the sphere. 

\end{example}

\section{\label{numerical 2} Numerical integration in Besov spaces}

The techniques of the previous section can also be used to study the error in
numerical integration from a different perspective. So far we have considered
the worst case error
\[
\left\{  {\displaystyle\int_{\mathbf{X}}} \sup_{\left\Vert
f\right\Vert _{\mathbb{H}_{p}^{\Phi}\left(  \mathcal{M}\right)  }\leqslant1}
\left\vert \mathcal{E}_{\mathbf{x},\mathbf{\omega}} \left(  f\right)
\right\vert ^{q}d\mathbf{x}\right\}  ^{1/q},
\]
whereas now we will estimate the error
\[
\left\{
{\displaystyle\int_{\mathbf{X}}} \left\vert \mathcal{E}_{\mathbf{x}%
,\mathbf{\omega}}\left(  f\right)  \right\vert ^{p} d\mathbf{x}\right\}
^{1/p}.
\]
for a given $f\in\dot{\mathbb{B}}_{p,\infty
}^{\varphi} \left(  \mathcal{M}\right)$. The following is a slightly generalized version
of Theorem III in the Introduction.



\begin{theorem}
\label{Thm Besov} Assume that a metric measure space $\mathcal{M}$ can be
decomposed into a finite disjoint union of sets in the form $\mathcal{M=X}%
_{1}\cup\cdots\cup\mathcal{X}_{N}$, with measure $0<\left\vert \mathcal{X}%
_{j}\right\vert =\omega_{j}<+\infty$ and $0<\operatorname*{diam}\left(
\mathcal{X}_{j}\right)  =\delta_{j}<+\infty$. Also let $\varphi\left(
t\right)  $ be a non-negative increasing function in $t\geqslant0$, and let
$\dot{\mathbb{B}}_{p,\infty}^{\varphi}\left(  \mathcal{M}\right)  $ be the
associated Besov space. Then for every $1\leqslant p\leqslant+\infty$,
\begin{equation}
\left\{  {\int_{\mathbf{X}}}\left\vert \mathcal{E}_{\mathbf{x},\mathbf{\omega
}}\left(  f\right)  \right\vert ^{p}d\mathbf{x}\right\}  ^{1/p}\leqslant
2\left\vert \mathcal{M}\right\vert ^{1-1/p}\varphi\left(  2\sup\left\{
\delta_{j}\right\}  \right)  \left\Vert f\right\Vert _{\dot{\mathbb{B}}_{p,\infty
}^{\varphi}\left(  \mathcal{M}\right)  }.\label{1}%
\end{equation}
Furthermore, if $1\leqslant p\leqslant2$, then
\begin{equation}
\left\{  {\int_{\mathbf{X}}}\left\vert \mathcal{E}_{\mathbf{x},\mathbf{\omega
}}\left(  f\right)  \right\vert ^{p}d\mathbf{x}\right\}  ^{1/p}\leqslant
2B\left(  p\right)  \sup\left\{  \omega_{j}^{1-1/p}\right\}  \varphi\left(
2\sup\left\{  \delta_{j}\right\}  \right)  \left\Vert f\right\Vert
_{\dot{\mathbb{B}}_{p,\infty}^{\varphi}\left(  \mathcal{M}\right)  },\label{2}%
\end{equation}
and if $2\leqslant p<+\infty$, then
\begin{equation}
\left\{  {\int_{\mathbf{X}}}\left\vert \mathcal{E}_{\mathbf{x},\mathbf{\omega
}}\left(  f\right)  \right\vert ^{p}d\mathbf{x}\right\}  ^{1/p}\leqslant
2B\left(  p\right)  \left\vert \mathcal{M}\right\vert ^{1/2-1/p}\sup\left\{
\omega_{j}^{1/2}\right\}  \varphi\left(  2\sup\left\{  \delta_{j}\right\}
\right)  \left\Vert f\right\Vert _{\dot{\mathbb{B}}_{p,\infty}^{\varphi}\left(
\mathcal{M}\right)  }.\label{3}%
\end{equation}

\end{theorem}

Observe that the estimate (\ref{1}) is of some interest only for large $p$.
Indeed if $1\leqslant p\leqslant2$, then (\ref{2}) is better than (\ref{1}),
and if $2\leqslant p<+\infty$ and $B\left(  p\right)  \leqslant\left\vert
\mathcal{M}\right\vert ^{1/2} \left(  \sup\left\{  \omega_{j}^{1/2}\right\}
\right)  ^{-1}$, then (\ref{3}) is better than (\ref{1}).

In order to prove Theorem \ref{Thm Besov}, we need a Poincar\'{e} type
inequality for functions in Haj\l asz--Besov spaces.

\begin{lemma}
\label{Poincare} Let $1\leqslant p\leqslant+\infty$, let $\mathcal{M}$ be a
metric measure space, and let $\left\{  g_{n}\left(  x\right)  \right\}  $ be
a $\varphi$-gradient for an integrable function $f\left(  x\right)  $. Let
$\mathcal{X}$ be a measurable subset of $\mathcal{M}$ with $\omega=\left\vert
\mathcal{X}\right\vert >0$ and $\operatorname*{diam}\left(  \mathcal{X}%
\right)  \leqslant2^{-n}$, and let
\[
f_{\mathcal{X}}=\frac{1}{\omega}{\displaystyle\int_{\mathcal{X}}}f\left(
y\right)  dy.
\]
Then
\[
\left\{  {\int_{\mathcal{X}}}\left\vert f\left(  x\right)  -f_{\mathcal{X}%
}\right\vert ^{p}\frac{dx}{\omega}\right\}  ^{1/p}\leqslant2\varphi\left(
2^{-n}\right)  \left\{  {\int_{\mathcal{X}}}\left\vert g_{n}\left(  x\right)
\right\vert ^{p}\dfrac{dx}{\omega}\right\}  ^{1/p}.
\]

\end{lemma}

\begin{proof}
For almost every $x$ and $y$ with $\left\vert x-y\right\vert \leqslant2^{-n}
$, we have
\[
\left\vert f\left(  x\right)  -f\left(  y\right)  \right\vert \leqslant
\varphi\left(  2^{-n}\right)  \left(  g_{n}\left(  x\right)  +g_{n}\left(
y\right)  \right)  .
\]
Then, by H\"{o}lder's inequality, we obtain
\begin{align*}
&  \left\{  {\int_{\mathcal{X}}}\left\vert f\left(  x\right)  -f_{\mathcal{X}%
}\right\vert ^{p}\frac{dx}{\omega}\right\}  ^{1/p}\leqslant\left\{
{\int_{\mathcal{X}}}{\int_{\mathcal{X}}}\left\vert f\left(  x\right)
-f(y)\right\vert ^{p}\dfrac{dx}{\omega}\frac{dy}{\omega}\right\}  ^{1/p}\\
&  \quad\leqslant\varphi\left(  2^{-n}\right)  \left\{  {\int_{\mathcal{X}}%
}{\int_{\mathcal{X}}}\left\vert g_{n}\left(  x\right)  +g_{n}\left(  y\right)
\right\vert ^{p}\frac{dx}{\omega}\frac{dy}{\omega}\right\}  ^{1/p}%
\leqslant2\varphi\left(  2^{-n}\right)  \left\{  {\int_{\mathcal{X}}%
}\left\vert g_{n}\left(  x\right)  \right\vert ^{p}\dfrac{dx}{\omega}\right\}
^{1/p}.
\end{align*}
 \end{proof}


\begin{proof}{\it (of Theorem~\ref{Thm Besov})}
Let $f\in\dot{\mathbb{B}}_{p,\infty}^{\varphi}\left(\mathcal{M}\right)$, and let
$\left\{g_{n}\left(x\right)\right\}$ be a $\varphi$-gradient for $f\left(x\right)$.
Choose $n$ such that $2^{-n-1}<\sup\delta_{j}\leqslant2^{-n}$.
Then by Lemma \ref{Poincare}, we have
\begin{align*}
&
\left\{
{\displaystyle\int_{\mathbf{X}}}
\left\vert
{\displaystyle\sum_{j}}
\omega_{j}f\left(x_{j}\right)-
{\displaystyle\int_{\mathcal{M}}}
f\left(x\right)dx\right\vert^{p}d\mathbf{x}\right\}^{1/p}
=\left\{
{\displaystyle\int_{\mathbf{X}}}
\left\vert
{\displaystyle\sum_{j}}
\omega_{j}\left(f\left(x_{j}\right)-f_{\mathcal{X}_{j}}\right)\right\vert^{p}
d\mathbf{x}\right\}^{1/p}
\\
&\quad
\leqslant
{\displaystyle\sum_{j}}
\omega_{j}\left\{
{\displaystyle\int_{\mathcal{X}_{j}}}
\left\vert f\left(x_{j}\right)-f_{\mathcal{X}_{j}}\right\vert^{p}
\frac{dx_{j}}{\omega_{j}}\right\}^{1/p}
\leqslant2\varphi\left(2^{-n}\right)
{\displaystyle\sum_{j}}
\omega_{j}\left\{
{\displaystyle\int_{\mathcal{X}_{j}}}
\left\vert g_{n}\left(x_{j}\right)\right\vert^{p}\dfrac{dx_{j}}{\omega_{j}}\right\}^{1/p}
\\
&\quad
\leqslant2\varphi\left(2^{-n}\right)\left\{
{\displaystyle\sum_{j}}
\omega_{j}\right\}^{1-1/p}\left\{
{\displaystyle\sum_{j}}
{\displaystyle\int_{\mathcal{X}_{j}}}
\left\vert g_{n}\left(x_{j}\right)\right\vert^{p}dx_{j}\right\}^{1/p}
\\
&\quad
\leqslant2\varphi\left(2\sup\delta_{j}\right)
\left\vert \mathcal{M}\right\vert^{1-1/p}\left\{
{\displaystyle\int_{\mathcal{M}}}
\left\vert g_{n}\left(x\right)\right\vert^{p}dx\right\}^{1/p}.
\end{align*}
The proofs of (\ref{2}) and (\ref{3}) are similar.
Indeed, by Corollary \ref{Cor-MZ}, we have
\begin{align*}
&
\left\{
{\displaystyle\int_{\mathbf{X}}}
\left\vert
{\displaystyle\sum_{j=1}^{N}}
\omega_{j}f\left(x_{j}\right)-
{\displaystyle\int_{\mathcal{M}}}
f(x)dx\right\vert^{p}d\mathbf{x}\right\}^{1/p}
\\
&\quad
\leqslant B\left(p\right)\left\{
{\displaystyle\int_{\mathbf{X}}}
\left(
{\displaystyle\sum_{j=1}^{N}}
\omega_{j}^{2}
\left\vert f\left(x_{j}\right)-f_{\mathcal{X}_{j}}\right\vert^{2}\right)^{p/2}
d\mathbf{x}\right\}^{1/p}.
\end{align*}
Choose $n$ such that $2^{-n-1}<\sup\delta_{j}\leqslant2^{-n}$, and assume
that $1\leqslant p\leqslant2$.
By Lemma \ref{Poincare}, we obtain
\begin{align*}
&
\left\{
{\displaystyle\int_{\mathbf{X}}}
\left(
{\displaystyle\sum_{j=1}^{N}}
\omega_{j}^{2}\left\vert f\left(x_{j}\right)
-f_{\mathcal{X}_{j}}\right\vert^{2}\right)^{p/2}d\mathbf{x}\right\}^{1/p}
\leqslant\left\{
{\displaystyle\int_{\mathbf{X}}}
\left(
{\displaystyle\sum_{j=1}^{N}}
\omega_{j}^{p}\left\vert f\left(x_{j}\right)
-f_{\mathcal{X}_{j}}\right\vert^{p}\right)d\mathbf{x}\right\}^{1/p}
\\
&\quad
=\left\{
{\displaystyle\sum_{j=1}^{N}}
\omega_{j}^{p}
{\displaystyle\int_{\mathcal{X}_{j}}}
\left\vert f\left(x_{j}\right)-f_{\mathcal{X}_{j}}\right\vert^{p}
\frac{dx_{j}}{\omega_{j}}\right\}^{1/p}
\leqslant2\varphi\left(2^{-n}\right)\left\{
{\displaystyle\sum_{j=1}^{N}}
\omega_{j}^{p}
{\displaystyle\int_{\mathcal{X}_{j}}}
\left\vert g_{n}\left(x_{j}\right)\right\vert^{p}
\frac{dx_{j}}{\omega_{j}}\right\}^{1/p}
\\
&\quad
\leqslant2\sup\left(\omega_{j}^{1-1/p}\right)
\varphi\left(2\sup\left\{\delta_{j}\right\}\right)
\left\{\int_{\mathcal{M}}\left\vert g_{n}\left(x\right)\right\vert^{p}dx\right\}^{1/p}.
\end{align*}
Similarly, if $2\leqslant p<+\infty$, then H\"{o}lder's inequality with indices
$p/\left(p-2\right)$ and $p/2$ yields%
\begin{align*}
&
\left\{
{\displaystyle\int_{\mathbf{X}}}
\left(
{\displaystyle\sum_{j=1}^{N}}
\omega_{j}^{2}\left\vert f\left(x_{j}\right)
-f_{\mathcal{X}_{j}}\right\vert^{2}\right)^{p/2}d\mathbf{x}\right\}^{1/p}
\\
&\quad
\leqslant\left\{
{\displaystyle\sum_{j=1}^{N}}
\omega_{j}^{\left(2p-2\right)/\left(p-2\right)}\right\}^{\left(p-2\right)/2p}\left\{
{\displaystyle\sum_{j=1}^{N}}
\omega_{j}
{\displaystyle\int_{\mathcal{X}_{j}}}
\left\vert f\left(x_{j}\right)-f_{\mathcal{X}_{j}}\right\vert^{p}
\frac{dx_{j}}{\omega_{j}}\right\}^{1/p}
\\
&\quad
\leqslant\sup\left(\omega_{j}^{1/2}\right)\left\{
{\displaystyle\sum_{j=1}^{N}}
\omega_{j}\right\}^{\left(p-2\right)/2p}2\varphi\left(2^{-n}\right)
\left\{
{\displaystyle\sum_{j=1}^{N}}
\omega_{j}
{\displaystyle\int_{\mathcal{X}_{j}}}
\left\vert g_{n}\left(x_{j}\right)\right\vert^{p}
\frac{dx_{j}}{\omega_{j}}\right\}^{1/p}
\\
&\quad
\leqslant2\left\vert \mathcal{M}\right\vert^{1/2-1/p}
\sup\left(\omega_{j}^{1/2}\right)\varphi\left(2\sup\delta_{j}\right)\left\{
{\displaystyle\int_{\mathcal{M}}}
\left\vert g_{n}\left(x\right)\right\vert^{p}dx\right\}^{1/p}.
\end{align*}

\end{proof}

\begin{corollary}
Assume that a metric measure space $\mathcal{M}$ can be decomposed into a
finite disjoint union of sets in the form $\mathcal{M=X}_{1}\cup\cdots
\cup\mathcal{X}_{N}$, with $\omega_{j}=\left\vert \mathcal{X}_{j}\right\vert
\approx N^{-1}$ and $\operatorname*{diam}\left(  \mathcal{X}_{j}\right)
\approx N^{-1/d}$ for a suitable positive constant $d$. Then for every
$1\leqslant p<+\infty$ and every $0<\varepsilon<1$, there exists a constant
$c$ with the following property. For every function $f\left(  x\right)  $ in
the Besov space $\dot{\mathbb{B}}_{p,\infty}^{\varphi}\left(  \mathcal{M}\right)  $,
$\varphi\left(  t\right)  =t^{\alpha}$ and $\alpha>0$, with probability
greater than $1-\varepsilon$, a random choice of points $\left\{
x_{j}\right\}  $ in $\left\{  \mathcal{X}_{j}\right\}  $ gives
\[
\left\vert {\sum_{j=1}^{N}}\omega_{j}f\left(  x_{j}\right)  -{\int%
_{\mathcal{M}}}f(x)dx\right\vert \leqslant\left\{
\begin{array}
[c]{ll}%
c\left\Vert f\right\Vert _{\dot{\mathbb{B}}_{p,\infty}^{\varphi}\left(
\mathcal{M}\right)  }N^{1/p-1-\alpha/d} & \text{if $1\leqslant p\leqslant2$%
},\\
c\left\Vert f\right\Vert _{\dot{\mathbb{B}}_{p,\infty}^{\varphi}\left(
\mathcal{M}\right)  }N^{-1/2-\alpha/d} & \text{if $2\leqslant p<+\infty$}.
\end{array}
\right.
\]

\end{corollary}

\begin{proof}
This follows from Theorem~\ref{Thm Besov} via Chebyshev's inequality.
 \end{proof}

The following example shows that Theorem \ref{Thm Besov} is essentially sharp.

\begin{example}
As in Theorem \ref{GigLeo}, let $\mathcal{M}$ be a metric measure space of
finite measure with the property that there exist positive constants $H$, $K$,
$d$, such that for every $y\in\mathcal{M}$ and $0<r<\operatorname*{diam}%
\left(  \mathcal{M}\right)  $,%
\[
Hr^{d}\leqslant\left\vert \left\{  x\in\mathcal{M}:\left\vert x-y\right\vert
<r\right\}  \right\vert \leqslant Kr^{d}.
\]
For every $N$, the space $\mathcal{M}$ can be decomposed into a finite
disjoint union of sets in the form $\mathcal{M}=\mathcal{X}_{1}\cup\cdots
\cup\mathcal{X}_{N}$, with $\omega_{j}=\left\vert \mathcal{X}_{j}\right\vert
\approx N^{-1}$ and $\delta_{j}=\operatorname*{diam}\left(  \mathcal{X}%
_{j}\right)  \approx N^{-1/d}$. Moreover every $\mathcal{X}_{j}$ contains a
ball $B\left(  w_{j},r_{j}\right)  $ of centre $w_{j}$ and radius
$r_{j}\approx\delta_{j}$. It is possible to prove that each ball $B\left(
w_{j},r_{j}\right)  $ contains two disjoint balls $B\left(  y_{j},\varepsilon
r_{j}\right)  $ and $B\left(  z_{j},\varepsilon r_{j}\right)  $, with
$\varepsilon=6^{-1}\left(  2K/H\right)  ^{-1/d}$. Fix $1\leqslant j\leqslant
N$ and define%
\[
f_{j}\left(  x\right)  =\left(  1-\frac{2}{\varepsilon r_{j}}\left\vert
x-y_{j}\right\vert \right)  _{+}-\vartheta_{j}\left(  1-\frac{2}{\varepsilon
r_{j}}\left\vert x-z_{j}\right\vert \right)  _{+},
\]
with $\vartheta_{j}$ such that $f$ has mean $0$. Observe that $0<\vartheta
_{j}<C$ independent of $N$. Also note that there exist constants $A$ and $B$
independent of $N$ such that%
\[
\int\left\vert f_{j}\left(  x\right)  \right\vert ^{p}\frac{dx}{\omega_{j}%
}\geqslant A,
\]
and%
\[
\left\vert f_{j}\left(  x\right)  -f_{j}\left(  y\right)  \right\vert
\leqslant B\delta_{j}^{-1}\left\vert x-y\right\vert ~~~~\text{for all }%
x,y\in\mathcal{M}\text{.}%
\]
If $\varphi\left(  t\right)  =t^{\alpha}$ with $d/p<\alpha\leqslant1$, and if
$w_{j}$ is the above defined point in $\mathcal{X}_{j}$, then the function%
\[
g_{j}\left(  x\right)  =c\min\left(  \delta_{j}^{-\alpha},\left\vert
x-w_{j}\right\vert ^{-\alpha}\right)
\]
is a $\varphi$-gradient of $f\left(  x\right)  $. Indeed, if $x\in
\mathcal{X}_{j}$ and $\left\vert x-y\right\vert \leqslant2\delta_{j}$, then%
\[
\left\vert f_{j}\left(  x\right)  -f_{j}\left(  y\right)  \right\vert
\leqslant B\delta_{j}^{-1}\left\vert x-y\right\vert \leqslant c\delta
_{j}^{-\alpha}\left\vert x-y\right\vert ^{\alpha},
\]
while for $x\in\mathcal{X}_{j}$ and $\left\vert x-y\right\vert >2\delta_{j} $,%
\begin{align*}
\left\vert f_{j}\left(  x\right)  -f_{j}\left(  y\right)  \right\vert  &
=\left\vert f_{j}\left(  x\right)  \right\vert \leqslant\max\left\{
1,\vartheta_{j}\right\}  =\max\left\{  1,\vartheta_{j}\right\}  \left\vert
x-y\right\vert ^{-\alpha}\left\vert x-y\right\vert ^{\alpha}\\
& \leqslant c\left\vert w_{j}-y\right\vert ^{-\alpha}\left\vert x-y\right\vert
^{\alpha}.
\end{align*}
In particular,%
\[
\left\Vert f_{j}\right\Vert _{\dot{\mathbb{B}}_{p,\infty}^{\varphi}\left(
\mathcal{M}\right)  }\leqslant c\left\{  \int_{\mathcal{M}}\min\left(
\delta_{j}^{-\alpha p},\left\vert x-w_{j}\right\vert ^{-\alpha p}\right)
dx\right\}  ^{1/p}\leqslant cN^{\alpha/d-1/p}.
\]
Moreover, since $f(x)$ has mean zero and it is supported in $\mathcal{X}_{j}$,%
\[
\left\{  \int_{\mathbf{X}}\left\vert {\sum_{k=1}^{N}}\omega_{k}f_{j}\left(
x_{k}\right)  -{\int_{\mathcal{M}}}f_{j}\left(  x\right)  dx\right\vert
^{p}d\mathbf{x}\right\}  ^{1/p}=\omega_{j}\left\{  \int_{\mathcal{X}_{j}%
}\left\vert f_{j}\left(  x_{j}\right)  \right\vert ^{p}\frac{dx_{j}}%
{\omega_{j}}\right\}  ^{1/p}\geqslant cN^{-1}.
\]
Finally,%
\[
N^{-1}=N^{1/p-1}N^{-\alpha/d}N^{\alpha/d-1/p}\geqslant c\sup\left\{
\omega_{j}^{1-1/p}\right\}  \varphi\left(  2\sup\left\{  \delta_{j}\right\}
\right)  \left\Vert f_{j}\right\Vert _{\dot{\mathbb{B}}_{p,\infty}^{\varphi}\left(
\mathcal{M}\right)  }.
\]
This shows that Theorem \ref{Thm Besov} with $1\leqslant p\leqslant2$ is sharp.

In order to show that the theorem is essentially sharp also for $2<p\leqslant
+\infty$, let $f_{j}\left(  x\right)  $ as before and define
\[
f\left(  x\right)  ={\sum_{j=1}^{N}}f_{j}\left(  x\right)  .
\]
If $\varphi\left(  t\right)  =t^{\alpha}$, then a $\varphi$-gradient of
$f\left(  x\right)  $ is given by
\[
g\left(  x\right)  =cN^{\alpha/d}.
\]
Indeed, if $x,y\in\mathcal{X}_{j}$, then
\[
\left\vert f\left(  x\right)  -f\left(  y\right)  \right\vert =\left\vert
f_{j}\left(  x\right)  -f_{j}\left(  y\right)  \right\vert \leqslant
B\delta_{j}^{-1}\left\vert x-y\right\vert \leqslant cN^{\alpha/d}\left\vert
x-y\right\vert ^{\alpha}.
\]
If $x\in\mathcal{X}_{i}$ and $y\in\mathcal{X}_{j}$, with $i\neq j$ and
$\left\vert x-y\right\vert \leqslant N^{-1/d}$, then
\[
\left\vert f\left(  x\right)  -f\left(  y\right)  \right\vert =\left\vert
f_{i}\left(  x\right)  -f_{j}\left(  y\right)  \right\vert \leqslant\left\vert
f_{i}\left(  x\right)  -f_{i}\left(  y\right)  \right\vert +\left\vert
f_{j}\left(  y\right)  -f_{j}\left(  x\right)  \right\vert \leqslant
cN^{\alpha/d}\left\vert x-y\right\vert ^{\alpha}.
\]
If $\left\vert x-y\right\vert \geqslant N^{-1/d}$, then
\[
\left\vert f\left(  x\right)  -f\left(  y\right)  \right\vert \leqslant
c\leqslant cN^{\alpha/d}\left\vert x-y\right\vert ^{\alpha}.
\]
This gives
\[
\left\Vert f\right\Vert _{\dot{\mathbb{B}}_{p,\infty}^{\varphi}\left(
\mathcal{M}\right)  }\leqslant cN^{\alpha/d}.
\]
Moreover, the Marcinkiewicz--Zygmund inequality gives for $2<p<+\infty$,
\begin{align*}
&  \left\{  \int_{\mathbf{X}}\left\vert {\sum_{j=1}^{N}}\omega_{j}f\left(
x_{j}\right)  -{\int_{\mathcal{M}}}f\left(  x\right)  dx\right\vert
^{p}d\mathbf{x}\right\}  ^{1/p}\geqslant A\left(  p\right)  \left\{
\int_{\mathbf{X}}\left(  {\sum_{j=1}^{N}}\left\vert \omega_{j}f_{j}\left(
x_{j}\right)  \right\vert ^{2}\right)  ^{p/2}d\mathbf{x}\right\}  ^{1/p}\\
&  \geqslant A\left(  p\right)  \left\{  \int_{\mathbf{X}}\left(  {\sum
_{j=1}^{N}}\left\vert \omega_{j}f_{j}\left(  x_{j}\right)  \right\vert
^{2}\right)  d\mathbf{x}\right\}  ^{1/2}\geqslant A\left(  p\right)  \left\{
{\sum_{j=1}^{N}}\omega_{j}\int_{\mathcal{X}_{j}}\left\vert f_{j}\left(
x_{j}\right)  \right\vert ^{2}dx_{j}\right\}  ^{1/2}\\
&  \geqslant A\left(  p\right)  \min\left\{  \omega_{j}^{1/2}\right\}
\left\{  \int_{\mathcal{M}}\left\vert f\left(  x\right)  \right\vert
^{2}dx\right\}  ^{1/2}\geqslant cN^{-1/2}.
\end{align*}
Finally,
\[
N^{-1/2}=N^{-1/2}N^{-\alpha/d}N^{\alpha/d}\geqslant c\sup\left\{  \omega
_{j}^{1/2}\right\}  \varphi\left(  2\sup\left\{  \delta_{j}\right\}  \right)
\left\Vert f\right\Vert _{\dot{\mathbb{B}}_{p,\infty}^{\varphi}\left(
\mathbb{T}^{d}\right)  }.
\]
In particular, this estimate shows that Theorem \ref{Thm Besov} with
$2<p<+\infty$ is sharp.

When $p=+\infty$ let $f\left(  x\right)  $ and $\left\{  y_{j}\right\}  $ as
before, so that $f\left(  y_{j}\right)  =1$. Then
\[
\left\vert {\sum_{j=1}^{N}}\omega_{j}f\left(  y_{j}\right)  -{\int%
_{\mathcal{M}}}f\left(  x\right)  dx\right\vert =1\geqslant c\varphi\left(
2\sup\left\{  \delta_{j}\right\}  \right)  \left\Vert f\right\Vert
_{\dot{\mathbb{B}}_{p,\infty}^{\varphi}\left(  \mathcal{M}\right)  }.
\]
In particular, this estimate shows that Theorem \ref{Thm Besov} with
$p=+\infty$ is sharp.
\end{example}

\section{\label{discr} Discrepancy}

The following result on the  expected value of the $p$th power of the discrepancy of a random set of points
 with respect to a fixed given set in a measure space extends a result in
\cite[Lemma 5]{C}. The result in the latter paper concerns the case of a compact convex 
set in the $d$-dimensional unit cube, and the proof
is based on the combinatorial argument described in the Introduction, which in our case
is replaced by the Marcinkiewicz-Zygmund inequality.

\begin{theorem}
\label{theo:discrepancy} Assume that a metric measure space $\mathcal{M}$ is
decomposed into a finite disjoint union of sets in the form $\mathcal{M=X}%
_{1}\cup\cdots\cup\mathcal{X}_{N}$, and call $\omega_{j}=\left\vert
\mathcal{X}_{j}\right\vert $ and $\delta_{j}=\operatorname*{diam}\left(
\mathcal{X}_{j}\right) $. Let $\mathcal{B}$ be a measurable subset of
$\mathcal{M},$ and let
\[
\psi_{\mathcal{B}}\left(  t\right)  =\left\vert \left\{  x\in\mathcal{B}%
:\operatorname*{dist}\left\{  x,\mathcal{M}\setminus\mathcal{B}\right\}
\leqslant t\right\}  \right\vert +\left\vert \left\{  x\in\mathcal{M}%
\setminus\mathcal{B}:\operatorname*{dist}\left\{  x,\mathcal{B}\right\}
\leqslant t\right\}  \right\vert .
\]

If $\mathcal{J}$ is the set of indices $j$ such that $\mathcal{X}_{j}$
intersects both $\mathcal{B}$ and its complement, then the following hold:

\begin{description}
\item[\textrm{(i)}] For every choice of points $\left\{  x_{j}\right\}  $ in
$\left\{  \mathcal{X}_{j}\right\} ,$
\[
\left\vert {\displaystyle\sum_{j=1}^{N}} \omega_{j}\chi_{\mathcal{B}}\left(
x_{j}\right)  -\left\vert \mathcal{B}\right\vert \right\vert \leq
\psi_{\mathcal{B}}\left(  \sup_{j\in\mathcal{J}}\left\{  \delta_{j}\right\}
\right)  .
\]

\item[\textrm{(ii)}] For every $1\leqslant p<+\infty,$
\begin{gather*}
\left\{  {\displaystyle\int_{\mathbf{X}}} \left\vert {\displaystyle\sum
_{j=1}^{N}} \omega_{j}\chi_{\mathcal{B}}\left(  x_{j}\right)  -\left\vert
\mathcal{B}\right\vert \right\vert ^{p}{d\mathbf{x}}\right\}  ^{1/p} \leq
B\left(  p\right)  \sqrt{\sup_{j\in\mathcal{J}}\left\{  \omega_{j}\right\}
\psi_{\mathcal{B}}\left(  \sup_{j\in\mathcal{J}}\left\{  \delta_{j}\right\}
\right)  }.
\end{gather*}

\end{description}
\end{theorem}

Observe that the right hand side in (ii) is better than the one in (i) when
\[
B\left(  p\right)  \leq\sqrt{\frac{\psi\left(  \sup_{j\in\mathcal{J}
}\left\{  \delta_{j}\right\}  \right)  }{\sup_{j\in\mathcal{J}}\left\{
\omega_{j}\right\}  }}.
\]
Also observe that
\[
\sup_{j\in\mathcal{J}}\left\{  \omega_{j}\right\}  \leq\sum_{j\in\mathcal{J}%
}\omega_{j}\leq\psi_{\mathcal{B}}\left(  \sup_{j\in\mathcal{J}}\left\{
\delta_{j}\right\}  \right).
\]
For a sufficiently refined decomposition $\{\mathcal X_j\}$ of the space, 
one should expect $\sup_{j\in\mathcal{J}}\left\{  \omega_{j}\right\}$ to be
much smaller than $\psi_{\mathcal{B}}\left(  \sup_{j\in\mathcal{J}}\left\{
\delta_{j}\right\}  \right)$.
Hence, for a fixed value of $p$, estimate (ii) is in general better than (i) as $N\to+\infty$.
On the other hand, recall that $B\left(  p\right)  \rightarrow+\infty$ as
$p\rightarrow+\infty$, hence, for a fixed decomposition of $\mathcal M$, estimate (i) wins
for sufficiently large values of $p$.

\begin{proof}
The proof of (i) is elementary. For every choice of $x_{j}%
\in\mathcal{X}_{j}$ one has
\[%
{\displaystyle\sum_{j=1}^N}
\omega_{j}\chi_{\mathcal{B}}\left(  x_{j}\right)  -\left|  \mathcal{B}\right|
=%
{\displaystyle\sum_{j=1}^N}
\left(  \omega_{j}\chi_{\mathcal{B}\cap\mathcal{X}_{j}}\left(  x_{j}\right)
-\left|  \mathcal{B}\cap\mathcal{X}_{j}\right|  \right)  .
\]
If $\mathcal{X}_{j}\subseteq\mathcal{B}$ or if $\mathcal{B}\cap\mathcal{X}%
_{j}=\emptyset$ then $\omega_{j}\chi_{\mathcal{B}\cap\mathcal{X}_{j}}\left(
x_{j}\right)  -\left\vert \mathcal{B}\cap\mathcal{X}_{j}\right\vert =0$.
Moreover, for every $j$,%
\[
\left\vert \omega_{j}\chi_{\mathcal{B}\cap\mathcal{X}_{j}}\left(
x_{j}\right)  -\left\vert \mathcal{B}\cap\mathcal{X}_{j}\right\vert
\right\vert \leq\omega_{j}.%
\]
Then, by
the triangle inequality,%
\begin{gather*}
\left\vert
{\displaystyle\sum_{j=1}^N}
\left(  \omega_{j}\chi_{\mathcal{B}\cap\mathcal{X}_{j}}\left(  x_{j}\right)
-\left\vert \mathcal{B}\cap\mathcal{X}_{j}\right\vert \right)  \right\vert
\leq%
{\displaystyle\sum_{j\in\mathcal J}}
\omega_{j}
\leq\psi_{\mathcal B}\left(  \sup_{j\in\mathcal{J}}\left\{  \delta_{j}\right\}
\right).
\end{gather*}
The proof of (ii) is similar, with the crucial difference that we replace the triangle inequality
with the Marcinkiewicz-Zygmund inequality (Corollary \ref{Cor-MZ}),
\begin{align*}
&\left\{
{\displaystyle\int_{\mathbf{X}}}
\left\vert
{\displaystyle\sum_{j=1}^N}
\left(  \omega_{j}\chi_{\mathcal{B}\cap\mathcal{X}_{j}}\left(  x_{j}\right)
-\left\vert \mathcal{B}\cap\mathcal{X}_{j}\right\vert \right)  \right\vert
^{p}{d\mathbf{x}}\right\}
^{1/p}\\
&\leqslant B\left(  p\right)  \left\{
{\displaystyle\sum_{j\in\mathcal{J}}}
\omega_{j}^{2}\right\}  ^{1/2}\leqslant B\left(  p\right)  \sup_{j\in\mathcal{J}%
}\left\{  \omega_{j}^{1/2}\right\}  \left\{
{\displaystyle\sum_{j\in\mathcal{J}}}
\omega_{j}\right\}  ^{1/2}\\
&\leqslant B\left(  p\right)\sqrt{  \sup_{j\in\mathcal{J}}\left\{  \omega_{j}%
\right\}  \psi_{\mathcal B}\left(  \sup_{j\in\mathcal{J}}\left\{  \delta_{j}\right\}\right)}.\,\,\,
\end{align*}
 \end{proof}

Theorems like the above are the main building block for the proof 
of the existence of point distributions with small $L^p$ discrepancy
with respect to given collections of subsets. A
very general result of this type is the following.

\begin{corollary}
\label{discrepanza_p}
Assume that a metric measure space $\mathcal{M}$ can be decomposed into a
finite disjoint union of sets in the form $\mathcal{M=X}_{1}\cup\cdots
\cup\mathcal{X}_{N}$, with $\omega_{j}=\left\vert \mathcal{X}_{j}\right\vert
\approx N^{-1}$ and $\delta_j=\operatorname*{diam}\left(  \mathcal{X}_{j}\right)
\approx N^{-1/d}$. Let ${\mathbb{G}}$ be a collection of measurable subsets of
$\mathcal{M}$ with the property that there exist positive constants $c$ and $\beta$
such that for all sets $\mathcal{G}\in{\mathbb{G}}$
\[
\psi_{\mathcal{G}}\left(  t\right)  =\left\vert \left\{  x\in{\mathcal{G}%
}:\operatorname*{dist}\left\{  x,\mathcal{M}\setminus{\mathcal{G}}\right\}
\leqslant t\right\}  \right\vert +\left\vert \left\{  x\in\mathcal{M}%
\setminus{\mathcal{G}}:\operatorname*{dist}\left\{  x,{\mathcal{G}}\right\}
\leqslant t\right\}  \right\vert \leqslant ct^{\beta}.
\]
Then for any  finite positive measure $\mu$ on any sigma algebra
on ${\mathbb{G}}$, and for every $1\leqslant p<+\infty$ there exists a constant $C$ and a choice
of points $\left\{  x_{j}\right\}  $ in $\left\{  \mathcal{X}_{j}\right\}  $
such that
\[
\left(  \int_{\mathbb{G}}\left\vert {\sum_{j=1}^{N}}\omega_{j}\chi
_{{\mathcal{G}}}\left(  x_{j}\right)  -\left\vert {\mathcal{G}}\right\vert
\right\vert ^{p}\,d\mu(\mathcal{G})\right)  ^{1/p}\leqslant CN^{-1/2-\beta
/2d}.
\]
\end{corollary}

\begin{proof}
By point {\rm (ii)} of Theorem \ref{theo:discrepancy},
\begin{align*}
&{\displaystyle\int_\mathbb{G}\int_{\mathbf{X}}}
\left\vert
{\displaystyle\sum_{j=1}^N}
\omega_{j}\chi_{\mathcal G}\left(  x_{j}\right)  -\left\vert {\mathcal G}%
\right\vert \right\vert ^{p}{d\mathbf {x}}d\mu(\mathcal G)\\
&\leqslant B\left(  p\right)^p  {\displaystyle\int_\mathbb{G}}\left({\sup_{j\in\mathcal{J}}\left\{  \omega
_{j}\right\}  \psi_{{\mathcal G}}\left(  \sup_{j\in\mathcal{J}}\left\{  \delta_{j}\right\}
\right)  }\right)^{p/2}d\mu(\mathcal G)\\
&\leqslant B\left(  p\right)^p  \mu(\mathbb {G})\left({\sup_{j\in\mathcal{J}}\left\{  \omega
_{j}\right\}  c\left(  \sup_{j\in\mathcal{J}}\left\{  \delta_{j}\right\}
\right)^\beta  }\right)^{p/2}.
\end{align*}
This implies that there exists an $\mathbf x\in\mathbf X$ such that 
the thesis of the theorem holds.
 \end{proof}

We emphasize that under the hypotheses of Corollary
\ref{GigLeo} the required decomposition exists, and it is always possible to take all $\omega_j$ equal to $|\mathcal M|N^{-1}$. 
The corollary  has several possible applications. We now examine a few particular cases,
starting with the isotropic discrepancy (the discrepancy with respect to convex sets)
in the unit cube $[0,1]^{d}$. 

\begin{corollary}
Let $1\leqslant p<+\infty$, and let $\mu$ be a finite positive measure on a
sigma algebra on the collection ${\mathcal{K}}^{d}_{u}$ of all convex sets of
the unit cube $[0,1]^{d}.$ For any integer $N$ there exists a distribution of
points $\{x_{j}\}_{j=1}^{N}$ in $[0,1]^{d}$ such that%

\[
\left( \int_{\mathcal{K}^{d}_{u}}\left\vert \frac1N{\sum_{j=1}^{N}}\chi_{{K}%
}\left(  x_{j}\right)  -\left\vert {K}\right\vert \right\vert ^{p}%
\,d\mu(K)\right) ^{1/p} \leqslant CN^{-1/2-1/2d}.
\]
\end{corollary}

\begin{proof}
It suffices to show that Corollary \ref{discrepanza_p} applies with $\beta=1$.
First of all, the unit cube can be decomposed into a 
finite disjoint union of sets in the form $\mathcal{M=X}_{1}\cup\cdots
\cup\mathcal{X}_{N}$, with $\omega_{j}=\left\vert \mathcal{X}_{j}\right\vert
= N^{-1}$ and $\delta_j=\operatorname*{diam}\left(  \mathcal{X}_{j}\right)
\approx N^{-1/d}$. See, for example, Theorem \ref{GigLeo}.
It then suffices to observe that
for all convex sets $K$ 
the uniform estimate $\psi_{K}(t)\leqslant4dt$ holds. Indeed, by the coarea formula
\[
\psi_K(t)=|\{x\in[0,1]^d:{\mathrm {dist}}(x,\partial K)\leqslant t\}|
=\int_{-t}^t |\partial K_u|_{d-1}du,
\]
where for $u>0$ we define $K_u=\{x\in K:{\mathrm {dist}}(x,\partial K)\geqslant u\}$
and for $u<0$ we define $K_u=\{x\in [0,1]^d:{\mathrm {dist}}(x, K)\leqslant |u|\}$,
and $|\cdot|_{d-1}$ is the $(d-1)$-dimensional Hausdorff measure.
It is a well known property of convex sets that all sets $K_u$ are convex
(see e.g. \cite[Chapter 3]{Schneider}.
By the Archimedean property of monotonicity with respect to inclusion 
of the measure of the boundary of convex sets (see \cite[Property 5, page 52]{BF}), from $K_u\subset [0,1]^d$
it follows $|\partial K_u|_{d-1}\leqslant |\partial[0,1]^d|_{d-1}=2d.$
 \end{proof}

As an explicit example of measure, one could consider any finite measure supported
on the translated, rotated and dilated copies of a fixed convex set.
Hence, this result includes and extends well known results on the $L^{p}$ discrepancy
with respect to discs, or other collections of sets with ``reasonable''
shapes (see for example \cite[Theorem 2D]{C}). It is interesting to observe 
that if one replaces the above $L^p$ norm with a supremum in $K\in\mathcal{K}^{d}_{u}$,
then the above result fails. Indeed Schmidt (see \cite{Sch})
proved that for any $N$ point distribution in the unit cube
there exists a convex set with discrepancy of order $N^{-2/(d+1)}$.

It is perhaps worth mentioning some results about measures on the space of
convex sets. It is well known that the set of nonempty convex compact subsets
of ${\mathbb{R}}^{d},$ let us call it $\mathcal{K}^{d},$ can be made into a
metric space by introducing the Hausdorff distance
\[
d_{H}(A,B)=\max\{\sup_{a\in A}\,\inf_{b\in B}|a-b|,\sup_{b\in B}\,\inf_{a\in
A}|a-b|\}.
\]
A large class of sigma finite Borel measures on $\mathcal{K}^{d},$ which are
positive on open sets of $\mathcal{K}^{d}$ and are invariant under rigid
motions, has been recently constructed by L. M. Hoffmann (\cite{Hof}).
Just to give a rough idea,
let $\{K_{n}\}_{n=1}%
^{+\infty}$ be an enumeration of all polytopes of ${\mathbb{R}}^{d}$ with
vertices in rational points, and let $\sum_{n=1}^{+\infty}\alpha_{n}<+\infty$
be a convergent series with positive terms. Then one can define the measure
\[
\mu=\sum_{n=1}^{+\infty}\alpha_{n}\delta_{K_{n}},
\]
where $\delta_{K_{n}}$ is the Dirac delta centered at $K_{n}$.
This measure is supported on rational polytopes, but is positive on open sets since these rational polytopes
are dense. A suitable 
clever modification can be made invariant under rigid motions.  Nevertheless,
it can be shown that there are more isometries of the space $\mathcal{K}^{d}$
than those coming from rigid motions of $\mathbb{R}^{d}$, and in fact it has
been showed by Bandt and Baraki (see \cite{BB}) that for $d>1$ there are no
nontrivial sigma finite measures on $\mathcal{K}^{d},$ that are invariant with
respect to all isometries of the whole space $\mathcal{K}^{d}$. Hence it seems that 
there is no ``natural'' measure on $\mathcal{K}^{d}$.

Similar results hold in compact Riemannian manifolds. 
For the sake of simplicity, we state here a result on the $L^p$ discrepancy associated to geodesic balls.
\begin{corollary}
Let $\mathcal M$ be a compact Riemannian manifold with injectivity radius $r_0$, and let $0<r_1<r_0$. Denote by 
$B(x,r)$ the geodesic ball centered at the point $x\in\mathcal M$ with radius $r$. Then
for any $1\leqslant p<+\infty$, for any finite positive measure $\mu$ on $\mathcal M\times (0,r_1 )$, and for any integer $N$, there exists a distribution of
points $\{x_{j}\}_{j=1}^{N}$ in $\mathcal M$ such that%
\[
\left( \iint_{\mathcal M\times(0,r_1)}\left\vert \frac{|\mathcal M|}N{\sum_{j=1}^{N}}\chi_{{B(x,r)}%
}\left(  x_{j}\right)  -\left\vert {B(x,r)}\right\vert \right\vert ^{p}%
\,d\mu(x,r)\right) ^{1/p} \leqslant CN^{-1/2-1/2d}.
\]
\end{corollary}

\begin{proof} As before, it suffices to show that Corollary \ref{discrepanza_p} applies with $\beta=1$.
Indeed $\mathcal M$ can be decomposed into a 
finite disjoint union of sets in the form $\mathcal{M=X}_{1}\cup\cdots
\cup\mathcal{X}_{N}$, with $\omega_{j}=\left\vert \mathcal{X}_{j}\right\vert
=|\mathcal M| N^{-1}$ and $\delta_j=\operatorname*{diam}\left(  \mathcal{X}_{j}\right)
\approx N^{-1/d}$. See, for example, Theorem \ref{GigLeo}.
Moreover,  there exists a positive constant $c$ 
such that for every geodesic ball $B(x,r)$ with $r<r_1$ and every $t>0$ one has
$
\psi_{B(x,r)}\left(  t\right)   \leqslant ct
$.
It clearly suffices to prove this for all $0<t\leqslant (r_0-r_1)/2.$ Indeed, by the triangle inequality,
\begin{align*}
 &\left\{  y\in{B(x,r)%
}:\operatorname*{dist}\left\{  y,\mathcal{M}\setminus{B(x,r)}\right\}
\leqslant t\right\}  \subset B(x,r)\setminus B(x,r-t), \\
 & \left\{ y\in\mathcal{M}%
\setminus{B(x,r)}:\operatorname*{dist}\left\{ y,{B(x,r)}\right\}
\leqslant t\right\} \subset \overline B(x,r+t)\setminus B(x,r).
\end{align*}
If $s\leqslant 0$ then we set $B(x,s)=\emptyset$. Finally, if $r<r_1$ and $t<(r_0-r_1)/2$, the 
exponential map diffeomorphically maps the
annulus $\overline B(x,r+t)\setminus B(x,r-t)$ 
into the tangent space in $x$, and by a uniform bound on the Jacobian of
the exponential map, its measure is bounded above
by $c((r+t)^d-\max\{0,r-t\}^d)\leqslant ct $.
 \end{proof}

We have already mentioned that the above results fail in general
when one replaces the $L^p$ norm with a supremum. Nevertheless, 
an extra hypothesis concerning the complexity of the collection of sets $\mathbb G$
allows to obtain the same upper bound in the supremum case too, up to a logarithmic transgression. 
\begin{theorem}
\label{infinity}
Let $\mathcal{M}$ be a metric measure space with
finite measure with the property that there exist positive constants $d$ and $c_{1}$, such that for every
sufficiently large $N$ there exists a partition $\mathcal{M}=\mathcal{X}%
_{1}\cup\cdots\cup\mathcal{X}_{N}$ with
$\left|  \mathcal{X}_{j}\right|  = \left|  \mathcal{M} \right| N^{-1}$ and
$
{\mathrm{diam}}\left(\mathcal{X}_{j}\right)\leqslant c_{1}N^{-1/d}.
$
Let ${\mathbb{G}}$ be a collection of measurable subsets of
$\mathcal{M}$ with the following two properties:
\begin{description}
\item[\textrm{(i)}] There exist positive constants $c_2$ and $\beta$
such that for all sets $\mathcal{G}\in{\mathbb{G}}$
\[
\psi_{\mathcal{G}}\left(  t\right)  =\left\vert \left\{  x\in{\mathcal{G}%
}:\operatorname*{dist}\left\{  x,\mathcal{M}\setminus{\mathcal{G}}\right\}
\leqslant t\right\}  \right\vert +\left\vert \left\{  x\in\mathcal{M}%
\setminus{\mathcal{G}}:\operatorname*{dist}\left\{  x,{\mathcal{G}}\right\}
\leqslant t\right\}  \right\vert \leqslant c_2t^{\beta}.
\]
\item[\textrm{(ii)}] There exist positive constants $c_3$ and $\gamma$ such that 
for all integers $N$ and for all distributions $\mathcal P$ of $N$ points in $\mathcal M$
there are at most $c_3N^\gamma$ equivalence classes in $\mathbb G$, where $\mathcal G, \mathcal G'$ in
$\mathbb G$ are defined to be equivalent if $\mathcal P\cap\mathcal G=\mathcal P\cap\mathcal G'$.
\end{description}
Then for any integer $N$ there exists a distribution of
points $\{z_{j}\}_{j=1}^{N}$ such that
\[
\sup_{\mathcal G\in\mathbb G}\left\vert \frac{|\mathcal M|}N{\sum_{j=1}^{N}}\chi_{{\mathcal G}%
}\left(  z_{j}\right)  -\left\vert {\mathcal G}\right\vert \right\vert\leqslant CN^{-1/2-\beta/2d}\sqrt{\log N}.
\]
\end{theorem}
A few words on the above condition {\textrm(ii)} are perhaps necessary. We say that 
$\mathbb G$ shatters a finite subset $\mathcal P$ of $N$ points of $\mathcal M$ if there are
exactly $2^N$ distinct intersections of sets of $\mathbb G$ with $\mathcal P$. 
The Vapnik Chervonenkis dimension, or VC-dimension, of $\mathbb G$ is the supremum of the sizes of all finite subsets of $\mathcal M$
that are shattered by $\mathbb G$. By Sauer lemma (see \cite{Sauer}), condition {\textrm(ii)} coincides with
asking that the collection $\mathbb G$ has finite VC-dimension.
For example, the collection of convex sets in $\mathbb R^d$ has infinite VC-dimension.
Indeed a set of $N$ points on a circle can be easily shattered with convex sets.
On the other hand, the collection of balls in $\mathbb R^d$ has VC-dimension
$d+1$ (see \cite[Chapter 5]{M} for an account on this subject).

The proof that we present here follows closely the lines of the classic result for discs in the 
unit square, as one can find in J. Matou\v{s}ek' s book \cite{M}.
\begin{proof} 
Let $M=N^{q}$, where $q$ is a positive integer that will be fixed later.
Consider two partitions of $\mathcal M$ as in the hypothesis. The first is composed
by the $N$ sets $\mathcal X_1,\ldots,\mathcal X_N$, 
and the second is composed by the $M$ sets $\mathcal Y_1,\ldots,\mathcal Y_M$.
For any $j=1,\ldots,N$, define $I_j=\{i=1,\ldots,M:\mathcal Y_i\cap\mathcal X_j\neq\emptyset\}$ and, for all $i\in I_j$ define 
\[
\mathcal {Y}_{j,i}=\mathcal X_j\cap\mathcal Y_i.
\]
Fix a point $y_{j,i}$ in any of the sets $\mathcal {Y}_{j,i}$. Clearly $\{\mathcal {Y}_{j,i}\}_{i\in I_j}$ forms a partition of $\mathcal X_j$
and $\mathcal {Y}_{j,i}\subset\{x\in\mathcal M:|x-y_{i,j}|<c_1M^{-1/d}\}$. 

For each $j=1,\ldots, N$, let us pick one point $q_j$ among all the points $y_{j,i}$ with  $i\in I_j$. This point $q_j$ is chosen
randomly with probability ${\mathbb P}[q_j=y_{j,i}]=N|\mathcal Y_{j,i}|/|\mathcal M|$, the choices being independent for distinct values of $j$.
The discrepancy of the point distribution $\{q_j\}_{j=1}^N$ with respect to a given $\mathcal G\in\mathbb G$ can be estimated
as follows
\begin{align*}
&\left\vert \frac{|\mathcal M|}N{\sum_{j=1}^{N}}\chi_{{\mathcal G}%
}\left(  q_{j}\right)  -  \left\vert {\mathcal G}\right\vert \right\vert\\
&\leqslant  \left\vert \frac{|\mathcal M|}N\sum_{j=1}^{N}\left(\chi_{\mathcal G}%
\left(  q_{j}\right)  - \frac N{|\mathcal M|}\sum_{i\in I_j}|\mathcal Y_{j,i}|\chi_{{\mathcal G}%
}\left(  y_{j,i}\right)\right)\right\vert
+\left\vert  {\sum_{j=1}^{N}}\sum_{i\in I_j}|\mathcal Y_{j,i}|\chi_{{\mathcal G}%
}\left(  y_{j,i}\right)  -\left\vert {\mathcal G}\right\vert \right\vert
\end{align*}
The second term in the above sum is deterministic and can be treated easily. 
By Theorem \ref{theo:discrepancy} (i) it is bounded above by
\[
\psi_{\mathcal G}\left(c_1M^{-1/d}\right)\leqslant c_2\left(c_1M^{-1/d}\right)^\beta= c_1^\beta c_2 N^{-\beta q/d}.
\]
It is therefore sufficient to take $q\geqslant  d/\beta$ to obtain an estimate better than what is needed.

The other term is of a probabilistic nature. We only need to consider the 
values of $j$ for which $\mathcal X_j$ intersects both $\mathcal G$ and $\mathcal M$.
Call this set $\mathcal J$ and its cardinality $m$. Since
\[
m\frac{|\mathcal M|}N\leqslant\psi_{\mathcal G}\left(c_1N^{-1/d}\right)\leqslant  c_1^\beta c_2 N^{-\beta/d},
\]
we have $m\leqslant c_1^\beta c_2 N^{1-\beta/d}/|\mathcal M|$.
Let us now set 
\[
k_j=\frac N{|\mathcal M|}\sum_{i\in I_j}|\mathcal Y_{j,i}|\chi_{{\mathcal G}%
}\left(  y_{j,i}\right),
\] 
and call $F_j$ the random variable
$
\chi_{\mathcal G}
\left(q_{j}\right)  - k_j.
$
Thus we have
\[
\frac{|\mathcal M|}N\sum_{j=1}^{N}\left(\chi_{\mathcal G}%
\left(  q_{j}\right)  - \frac N{|\mathcal M|}\sum_{i\in I_j}|\mathcal Y_{j,i}|\chi_{{\mathcal G}%
}\left(  y_{j,i}\right)\right)=\frac{|\mathcal M|}N\left(\sum_{j\in\mathcal J}F_j\right).
\]
The variables $F_j$ are mutually independent, and $F_j$ takes the value
$1-k_j$ with probability $k_j$, and $-k_j$ with probability $1-k_j$. Therefore 
for every $\Delta >0$ we have
\[
{\mathbb P}\left[\left|\sum_{j\in\mathcal J}F_j\right|\geqslant\Delta\right]\leqslant 2\exp(-2\Delta^2/m)
\]
(see, for example,
\cite[Corollary A.1.7]{AS}).
Let us fix a constant $C>0$. Then we have showed that
\begin{align*}
&{\mathbb P}\left[\frac{|\mathcal M|}N\left|\sum_{j\in\mathcal J}F_j\right|\geqslant CN^{-1/2-\beta/(2d)}\sqrt{\log N}\right]\\
&={\mathbb P}\left[\left|\sum_{j\in\mathcal J}F_j\right|\geqslant  CN^{1/2-\beta/(2d)}\sqrt{\log N}/|\mathcal M|\right]\\
&\leqslant 2\exp(-2C^2N^{1-\beta/d}\log N/|\mathcal M|^2)/m)\\
&\leqslant 2N^{-C^2c_1^{-\beta} c_2^{-1}|\mathcal M|^{-1}}.
\end{align*}
Finally, if $\mathbb F\subset \mathbb G$ contains 
one representative for each equivalence class, then
\begin{align*}
&{\mathbb P}\left[\frac{|\mathcal M|}N\left|\sum_{j\in\mathcal J}F_j\right|\geqslant  CN^{-1/2-\beta/(2d)}\sqrt{\log N}\text{ for some }\mathcal G\in\mathbb G\right]\\
&={\mathbb P}\left[\frac{|\mathcal M|}N\left|\sum_{j\in\mathcal J}F_j\right|\geqslant  CN^{-1/2-\beta/(2d)}\sqrt{\log N}\text{ for some }\mathcal G\in\mathbb F\right]\\
&\leqslant\sum_{\mathcal G\in\mathbb F}{\mathbb P}\left[\frac{|\mathcal M|}N\left|\sum_{j\in\mathcal J}F_j\right|\geqslant  CN^{-1/2-\beta/(2d)}\sqrt{\log N}\right]\\
&\leqslant 2c_3N^{\gamma-C^2c_1^{-\beta} c_2^{-1}|\mathcal M|^{-1}}<1
\end{align*}
if $C$ is large enough. The theorem follows.
 \end{proof}

The next Corollary shows one possible application of the above theorem.
\begin{corollary}
Let $\mathcal M$ be a $d$-dimensional compact Riemannian manifold isometrically embedded in $\mathbb R^D$,
and call $B^D(x,r)=\{y\in\mathbb R^D:\|y-x\|<r\}$ the Euclidean $D$-dimensional ball
of center $x$ and radius $r$.
Then there exist positive constants $r_0$ and $C$ such that for any integer $N$ there exists a distribution of
points $\{z_{j}\}_{j=1}^{N}\in\mathcal M$ such that
\[
\sup_{r\leqslant r_0, x\in\mathcal M}\left\vert \frac{|\mathcal M|}N{\sum_{j=1}^{N}}\chi_{{B^D(x,r)\cap\mathcal M}%
}\left(  z_{j}\right)  -\left\vert {B^D(x,r)\cap\mathcal M}\right\vert \right\vert\leqslant CN^{-1/2-1/2d}\sqrt{\log N}.
\]

\end{corollary}

Notice that, by the Nash embedding theorem, every Riemannian manifold can be isometrically
embedded into some Euclidean space. In particular, if one takes  as $\mathcal M$ the $d$-dimensional unit sphere in $\mathbb R^{d+1}$,
then the sets  $B^{d+1}(x,r)\cap\mathcal M$ of the above corollary coincide with the usual spherical caps, and one 
recovers Beck's estimate for the spherical cap discrepancy (see \cite[Theorem 24D]{BCbook}).
\begin{proof}
It is enough to show that the collection of subsets of the form $B^D(x,r)\cap\mathcal M$
satisfies the two hypotheses of Theorem \ref{infinity}.
By compactness, there exists a positive number $r_0$ such that for all
$0<r\leqslant r_0$ and for all $x\in\mathcal M$, the set $\mathcal N=\{y\in\mathcal M:\|x-y\|=r\}$
is a hypersurface of $\mathcal M$ with uniformly bounded $(d-1)$-dimensional volume. 
Furthermore, the measure 
of the set of points of $\mathcal M$ with geodesic distance from 
$\mathcal N$ less than or equal to $t$ is bounded above by
\[
\int_{\mathcal N}\int_{-t}^t|f(s,n)|dsdn,
\]
where $dn$ is the $(d-1)$-dimensional volume form on $\mathcal N$ and 
$f(s,n)$ is the (uniformly bounded) Jacobian of the exponential map of the normal bundle
on $\mathcal N$ in $\mathcal M$ (see \cite{HK} for the details). Thus
\[
\psi_{\mathcal M\cap B^D(x,r)}(t)\leqslant ct,
\]
and the first hypothesis of Theorem \ref{infinity} holds with $\beta=1$.
Finally, as we mention before, balls in $\mathbb R^D$
satisfy the second hypotheses of the same theorem with $\gamma=D+1$
(\cite[Chapter 5] {M}).
 \end{proof}

\section*{acknowledgements}
The authors wish to thank Dmitriy Bilyk for several conversations concerning
the results on discrepancy contained in this paper.


\end{document}